\documentclass[]{elsarticle}

\usepackage{amsmath,amssymb,amsfonts,mathabx,setspace}
\usepackage{graphicx,caption,epsfig,subfigure,epsfig,wrapfig}
\usepackage{url,color,verbatim,algorithmic}
\usepackage{booktabs}  

\usepackage{titlesec}
\titleformat{\section}  {\normalfont\fontsize{12}{15}\bfseries}{\thesection}{1em}{}
\titleformat{\subsection} {\normalfont\fontsize{10}{12}\bfseries}{\thesubsection}{1em}{}

\def\R{\mathbb{R}}

\def\E{\mathbb{E}}
\def\thetaOU{\theta}
\def\lambdaG{\lambda}

\newcommand{\wbar}\widebar

\newcommand{\cov}{\mathrm{cov}}

\definecolor{mygrey}{gray}{0.75}

\newcommand{\mbf}[1]{\boldsymbol{#1}}

\newcommand{\innerp}[2]{\langle #1,#2 \rangle}
\newcommand{\inp}[1]{\langle{#1}\rangle}

\newcommand{\realR}[1]{\mathbb{R}^{#1}}

\newcommand{\br}{\mbf{r}}

\newcommand{\bx}{\mbf{x}}
\newcommand{\bB}{\mbf{B}}
\newcommand{\bX}{\mbf{X}}
\newcommand{\bY}{\mbf{Y}}

\newcommand{\by}{\mbf{y}}

\newcommand{\btheta}{\mbf{A}}

\def\pV{J}

\newcommand{\intkernel}{\phi}

\newcommand{\intkernelvar}{\varphi}

\newcommand{\hypspace}{\mathcal{H}}

\newcommand{\argmax}[1]{\underset{#1}{\operatorname{arg}\operatorname{max}}\;}

\DeclareMathAlphabet{\mathpzc}{OT1}{pzc}{m}{it}

 \usepackage{soul}
\usepackage[dvipsnames]{xcolor}

\newcommand{\FL}[1]{\textcolor{blue}{{#1}}}
\newcommand{\red}[1]{\textcolor{red}{{#1}}}
\newcommand{\gray}[1]{\textcolor{gray}{{#1}}}

\newtheorem{theorem}{Theorem}

\newtheorem{definition}[theorem]{Definition}

\newtheorem{lemma}[theorem]{Lemma}

\newtheorem{proposition}[theorem]{Proposition}
\newtheorem{remark}[theorem]{Remark}

\newenvironment{proof}[1][Proof]{\noindent\textbf{#1.} }{\ \rule{0.5em}{0.5em}}

\numberwithin{equation}{section}
\numberwithin{theorem}{section}

\newcommand\RR{{\mathbb R}}
\newcommand\NN{{\mathbb N}}
\newcommand\CC{{\mathbb C}}

\journal{Journal of \LaTeX\ Templates}
\usepackage[margin=1in]{geometry}

\begin{document}

\begin{frontmatter}

\title{On the identifiability of interaction functions\\ in systems of interacting particles}


\author[address1]{Zhongyang Li} \ead{zhongyang.li@uconn.edu}
\author[address2,address3]{Fei Lu\corref{mycorrespondingauthor}} \ead{feilu@math.jhu.edu}
\author[address2,address3]{Mauro Maggioni} \ead{mauromaggionijhu@icloud.com}
\author[address4]{Sui Tang}\ead{suitang@math.ucsb.edu}
\author[address5]{Cheng Zhang}\ead{czhang77@ur.rochester.edu}


\cortext[mycorrespondingauthor]{Corresponding author.}
\address[address1]{Department of Mathematics,  University of Connecticut, Storrs, CT, USA 06269}
\address[address2]{Department of Mathematics, Johns Hopkins University, Baltimore, MD, USA, 21218}
\address[address3]{Department of Applied Mathematics and Statistics, Johns Hopkins University, Baltimore, MD, USA, 21218}
\address[address4]{Department of Mathematics,  University of California Santa Barbara, Isla Vista, USA 93117}
\address[address5]{Department of Mathematics, University of Rochester, Rochester, NY 14627}

\begin{abstract}
We address a fundamental issue in the nonparametric inference for systems of interacting particles: the identifiability of the interaction functions. We prove that the interaction functions are identifiable for a class of first-order stochastic systems, including linear systems with general initial laws and nonlinear systems with stationary distributions. We show that a coercivity condition is sufficient for identifiability and becomes necessary when the number of particles approaches infinity. The coercivity is equivalent to the strict positivity of related integral operators, which we prove by showing that their integral kernels are strictly positive definite by using M\"untz type theorems.
\end{abstract}

\end{frontmatter}
\vspace{-6mm}
\tableofcontents
\section{Introduction}

Dynamical systems of interacting particles or agents are widely used in many areas in science and engineering, such as physics \cite{DOCBC2006}, biology \cite{bernoff2013nonlocal}, social science \cite{motsch2014heterophilious,BT2015}; we refer to \cite{vicsek2012collective,carrillo2017review} for reviews. With the recent advancement of technology in data collection and computation, inference of such systems from data has attracted increasing attention \cite{herbert2011inferring, casadiego2017model, huang2019learning}. In general, such systems are high-dimensional and there is no natural parametric form for the interaction laws, so their inference tends to be statistically and computationally infeasible due to the curse of dimensionality. When the particles interact according to a function that depends only on pairwise distances, one only needs to estimate such interaction function, opening the possibility of statistically and computationally efficient inference techniques \cite{BFHM17,LZTM19, LMT19}. However, a fundamental challenge arises: the interaction function may be non-identifiable, because its values are under-determined from the observation data consisting of trajectories. To ensure the identifiability of the interaction function, a coercivity condition is introduced in \cite{BFHM17,LZTM19,LMT19}. In this study, we show that the coercivity condition is sufficient for the identifiability, that it becomes necessary when the number of particles goes to infinity, and that it holds true for linear systems and for a class of three-particle nonlinear systems with a stationary distribution.

More precisely, we consider a first-order stochastic gradient system of interacting particles in the form
\begin{equation}\label{eq:sys1st}
\left\{
\begin{aligned}
  d{\bX_i^t}& =\frac{1}{N} \sum_{1\leq j\leq N, j\neq i} \phi(|\bX_{j}^t - \bX_i^t|) \frac{\bX_{j}^t - \bX_i^t}{|\bX_{j}^t - \bX_i^t|}dt+ \sigma d\bB_i^t, \quad \text{for $i = 1, \ldots, N$},\\
  \bX^0&\sim \mu_0,
\end{aligned}
\right.
\end{equation}
where $\bX_i^t \in \realR{d}$ represents the position of particle $i$ at time $t$, $\{\bB_i^t\}_{i=1}^N$ are independent Brownian motions in $\realR{d}$ representing the random environment, $|\cdot|$ denotes the Euclidean norm, $\sigma>0$ is the strength of the noise. Without loss of generality, we assume $\sigma=1$ in \eqref{eq:sys1st}.
The function $\phi:\R^+=[0,\infty)\to \R$ models the pairwise interaction between particles, which is referred to as the \emph{interaction function}.
We assume that the initial condition $\bX^0$ has an exchangeable absolutely continuous distribution $ \mu_0$ on the state space $\R^{dN}$, that is, the joint distribution of $(\bX^0_{i_1},\ldots,\bX^0_{i_N})$ is $\mu_0$ for any ordering of the index set $\{i_1,\ldots,i_N\}= \{1,\ldots, N\}$. As a consequence, combining with the fact that the system is equivalent under permutations of the indices of the agents,
the distribution of $\bX^t$ is exchangeable for any $t\in [0,T]$.

We consider the identifiability of the interaction function $\phi$ from many independent trajectories on a time interval $[0,T]$, denoted by $\{\bX^{[0,T],m}\}_{m=1}^M$, in the likelihood-based nonparametric inference setting. We focus on the case of infinitely many trajectories (i.e. $M=\infty$) for our analysis. Clearly, the function space for inference must depend on the information from the process defined by the system \eqref{eq:sys1st}, and in particular we can hope to estimate the interaction function only on the interval explored by pairwise distances.  A natural choice is the space  $L^2(\wbar \rho_T)$ (or a subspace thereof), where $\wbar \rho_T$ is the average-in-time distribution of all the pairwise distances $\{|\bX_i^t-\bX_{j}^t|, t\in [0,T]\}_{i,j=1}^N$. By the exchangeability of the distribution of $\bX^t$, the distribution $\rho_t$ of $|\bX_i^t-\bX_{j}^t|$ is the same for all $(i,j)$ pairs (which is why we may abuse notation and avoid writing $\rho_{t,i,j}$), thus $\wbar \rho_T$ can be written as
 \begin{equation} \label{eq:rhoavg}
  \wbar{\rho}_T(dr) : = \frac{1}{T} \int_0^T \rho_t(dr) dt, \quad\text{ with }\quad \rho_t(dr) : =  \E[\delta(|\bX_i^t-\bX_{j}^t|\in dr)],
 \end{equation}
In other words, $\wbar \rho_T$ is the average of the measures $\{\rho_t, t\in[0,T] \}$. Note that $\wbar \rho_T$ depends on both the initial distribution and the interaction function $\phi$.

We define the identifiability of the interaction function as follows.

\begin{definition}[Identifiability]\label{def:identifiability}
The interaction function $\phi$ of the system \eqref{eq:sys1st}, which defines the process $\bX^{[0,T]}$, is said to be identifiable in a linear subspace $\hypspace$ of $L^2(\wbar\rho_T)$, if it is the unique maximizer of the expectation of the log-likelihood ratio of the process.
\end{definition}

In practice, the above identifiability requires a unique global maximizer for the expectation of the log-likelihood ratio (a functional on the high- or infinite-dimensional subspace $\hypspace$, see Section \ref{sec:inference} for details), which is difficult to verify. The following coercivity condition, introduced in  \cite{BFHM17,LZTM19,LMT19,LMT20}, provides an appealing alternative because it can be numerically verified from data. It ensures the uniqueness of the maximizer of the empirical likelihood ratio on finite dimensional hypothesis spaces by ensuring the Hessian to be strictly negative definite.

\begin{definition}[Coercivity condition on a time interval] \label{def_coercivty}
The  system \eqref{eq:sys1st} on $[0,T]$ 
 is said to satisfy a coercivity condition on a finite-dimensional linear subspace $\mathcal{H}\subset L^2(\bar \rho_T)$ with $\bar \rho_T$ defined in \eqref{eq:rhoavg} if
\begin{equation} \label{eq:coercivityDef}
c_{\mathcal{H},T} : = \quad \inf_{h\in  \mathcal{H}, \, \|h\|_{L^2(\wbar \rho_T)}=1}\frac{1}{T}\int_{0}^T \E[h(|\br_{12}^t|)h(|\br_{13}^t|)\frac{\innerp{\br_{12}^t}{\br_{13}^t}}{|\br_{12}^t||\br_{13}^t|} ]dt >0,
\end{equation}
where $\br_{ij}^t = \bX^{t}_i- \bX^t_j$. When $\hypspace\subseteq L^2(\wbar \rho_T)$ is infinite-dimensional, we say that the system satisfies a coercivity condition on $\hypspace$ if the coercivity condition holds on each finite dimensional linear subspace of $\mathcal{H}$.
\end{definition}

The coercivity condition on $\hypspace$ defined here is slightly different than the previous one in \cite{LZTM19,LMT19,LMT20}, which requires  $c_{\hypspace,T}+\frac{1}{N-2}>0$ and ensures the uniform concavity of the expectation of the log-likelihood ratio. This new definition has the advantage of being independent of $N$, and requires a positive coercivity constant $c_{\hypspace,T}$ only on each finite-dimensional hypothesis space $\hypspace$, rather than on any compact set of $L^2(\bar \rho_T)$, making it suitable for studying the mean-field limit when $N\to \infty$.
This new definition also highlights the dependence on the joint distribution of $(\br_{12}^t,\br_{12}^t)$ only, and the connection with positive integral operators (see Section \ref{sec:operator}). A drawback is that it can be slightly more restrictive than the previous one for finite $N$, with such difference vanishing as $N\rightarrow\infty$ (see Remark \ref{rmk:differentCC} for details).

We show that the coercivity condition is sufficient for the identifiability, and it holds for certain classes of interaction functions,  including $\phi(r) = r^{2\beta-1}$ with $\beta \in [\frac{1}{2},1]$:
\begin{theorem}\label{thm_mainAll}
Consider the system  \eqref{eq:sys1st} on $[0,T]$ with interaction function $\phi$, and the average distribution of the pairwise distances between particles $\wbar \rho_T$ as in \eqref{eq:rhoavg}.
\begin{itemize}
\item[(a).] The interaction function $\phi$ is identifiable in a linear subspace $\hypspace$ of $L^2(\wbar\rho_T)$ if the coercivity condition holds on $\hypspace$.
\item[(b).] The coercivity condition holds on $L^2(\wbar\rho_T)$ if $\phi(r) = \theta r$, i.e. when the system is linear, and the initial distribution $\mu_0$ of $\bX^0$ is a non-degenerate exchangeable Gaussian.
\item[(c).] The coercivity condition holds on $L^2(\wbar\rho_T)$ for nonlinear systems with three particles and with the following interaction functions and initial distribution:
\begin{enumerate}
 \item the interaction function $\phi$ is of the form
\begin{equation}\label{eq:Phi_nonlinear}
\phi(r) = \Phi'(r), \quad \text{ where } \Phi(r) :=  ar^{2\beta} + \Phi_0(r), \quad a>0,\, \beta \in[\frac{1}{2},1],
\end{equation}
where $\Phi_0 \in C^2(\R^+,\R)$ satisfies that $f(u,v): =\Phi_0(|u-v|): \R^d\times \R^d\to \R$ is a negative definite function and that $\lim_{r\to \infty}\Phi(r)=+\infty$;
    \item the joint probability density of $(\bX^0_1-\bX^0_2,\bX^0_1-\bX^0_3)$ is (with $Z$ being a normalizing constant)
\begin{align}\label{eq:r_invPDF}
p(u,v) = \frac{1}{Z} e^{-\frac{1}{3}\left[\Phi(|u|) + \Phi(|v|) + \Phi(|u-v|)\right], }
\end{align}
 i.e., an invariant density of the process $(\bX^t_1-\bX^t_2,\bX^t_1-\bX^t_3)$.
\end{enumerate}
\end{itemize}
\end{theorem}

Part (a) of Theorem \ref{thm_mainAll} is proved in Proposition \ref{propID_CC}. In addition to being sufficient for identifiability, the coercivity condition also becomes necessary when $N$, the number of particles in the system, is infinity. In particular, the coercivity constant $c_{\hypspace, T}$ in \eqref{eq:coercivityDef} is independent of $N$ and it depends only on the distribution of the process $(\bX^t_1-\bX^t_2,\bX^t_1-\bX^t_3)$. We prove Part (b) in Theorem \ref{thm:OU_non-radial} and Part (c) in Theorem \ref{Thm:cc_general}.

We  show that the coercivity condition is equivalent to the strict positivity of an integral operator arising from the expectation in \eqref{eq:coercivityDef} (see Section \ref{sec:operator}).
Then, to prove the strict positivity of the operator, we show that its integral kernel is strictly positive-definite, by introducing a series representation of the integral kernel and resorting to M\"untz-type theorems for the completeness of polynomials in $L^2(\wbar \rho_T)$  (see Section \ref{sec:OU}). In particular, in the treatment of nonlinear systems, we develop a ``comparison to a Gaussian kernel'' technique (Section \ref{sec:powerFn}-\ref{sec:genFn}) to prove the strictly positive-definiteness of integral kernels.

\bigskip
This study serves as a starting point towards understanding the identifiability of the interaction function for particle or agent systems. While providing a full characterization for linear systems, i.e., the coercivity condition holds for general initial distributions, Theorem \ref{thm_mainAll} provides limited results for nonlinear systems, covering only stationary initial distribution for systems with $N=3$ particles and with polynomial dominated interaction functions. The constraint $N=3$ arises because our series representation of the integral kernel is based on the explicit expression of the joint distribution of $(\br_{12}^t, \br_{13}^t)$, which is currently unknown to us when $N>3$, albeit we are hopeful to eventually be able to remove this constraint in future work. The constraint of stationary initial distribution may be removable by perturbation-type arguments. Future directions of research include, to name just a few, first-order nonlinear systems with more general interaction functions that are regular \cite{huang2019learning, LZTM19} or singular \cite{liu2016propagation,li2019mean}, second-order systems and systems with multiple types of particles or agents \cite{LMT19}, and mean-field equations \cite{meleard1996asymptotic,cattiaux2008probabilistic,JW17}.

Positive-definite integral kernels play an increasingly prominent role in many applications in science, in particular in statistical learning theory and in reproducing kernel Hilbert space (RKHS) representations \cite{cucker2002mathematical, smale2009geometry, Fasshauer11}. As a by-product, our results lead to a new class of positive-definite integral kernels from particle systems, and our technique of comparison to a Gaussian kernel may be of broader interest, for example in establishing identifiability of statistical learning problems.

The organization of the paper is as follows: we summarize in Table \ref{tab:notation} the  frequently used notations.  In Section \ref{sec:ccPD}, we introduce the coercivity conditions in inference, and establish the connections between identifiability, the coercivity condition and positive integral operators.
In Section \ref{sec:OU} we prove the coercivity condition for linear systems and Section \ref{sec:3particleSys} is devoted to a class of three-particle nonlinear systems with stationary distributions. We list in Section \ref{sec:append} the preliminaries, such as properties of positive-definite kernels, a M\"untz-type theorem on the half-line, and a stationary measure for gradient systems.
 \begin{table}[!h] 	\vspace{-1mm}
	\begin{center}
		\caption{  Notations  } \label{tab:notation}	\vspace{-2mm}
		\begin{tabular}{ l  l }
		\toprule 
			Notation   &  Description \\  \hline
			$\phi$ and $\varphi$    &  the true and, respectively, a generic interaction function       \\
			$\Phi$                          & the true interaction potential, such that $\Phi'(r) = \phi(r)$ as in \eqref{eq:Phi_nonlinear}\\
			$\bX^t_i$ and $\bX^{[0,T]}$    &  position of the $i$-th particle at time $t$ and, resp., trajectory on $[0,T]$       \\
	                $\br_{ij}^t= \bX^t_i-\bX^t_j $ & position difference from particle $j$ to particle $i$ at time $t$\\
	                $\rho_t$  and $\wbar\rho_T $ & probability distribution of $|\br_{12}^t|$ and, resp., its average on [0,$T$] in \eqref{eq:rhoavg}  \\
	                $L^2(\rho_t)$ and $L^2(\wbar \rho_T)$ & the function spaces $L^2(\R^+,\rho_t)$ and $L^2(\R^+,\wbar \rho_T)$  \\
	               	$p_t(u,v)$   and $p(u,v)$        & the joint density of $(\br_{12}^t, \br_{13}^t)$  and, resp., the stationary density, as in \eqref{eq:r_invPDF} \\
			\bottomrule	
		\end{tabular}
	\end{center}
	\vspace{-3mm}
\end{table}

\section{The coercivity conditions and strictly positive integral operators} \label{sec:ccPD}
In the context of likelihood-based nonparametric inference of the interaction function, we show that the coercivity condition is sufficient for identifiability, and it is equivalent to the strict positive-definiteness of an integral operator. Also, we introduce a coercivity condition at a single time, which suggests that the interaction function can be identifiable from many samples at a single time.

 In vector format, we can write the system \eqref{eq:sys1st} as
 \begin{align} \label{eq:sys_grad}
 d\bX^t&= - \nabla \pV_{\intkernel}(\bX^t)dt +  d\bB^t
 \end{align}
 where $\bX^t:=(\bX_i^t)_{i=1}^{N} \in \R^{Nd}$, and the potential function $\pV_{\intkernel}:\R^{Nd} \to \R$ is
\begin{align} \label{eq:potential}
 \pV_{\phi}(\bx) = \frac{1}{2N} \sum_{i, j=1, j\neq i}^N \Phi(|\bx_i-\bx_j|), \quad \bx\in \R^{Nd}, \text{ with } \Phi'(r)=\intkernel(r).
\end{align}

Since the pairwise potential $\Phi$ in \eqref{eq:Phi_nonlinear} is $C^2(\R^+)$
and $\lim_{r\to \infty}\Phi(r)=\infty$, the drift term in \eqref{eq:sys_grad} is locally Lipschitz and the total potential $J_\phi$ is a Lyapunov function for the system. Thus, a global solution exists (see e.g. \cite[Theorem 1.3]{Khasminskii12}). For  the existence and properties of the solutions in systems with singular potentials, we refer to  \cite{albeverio2003_StrongFeller,liu2016propagation,li2019mean} and the reference therein.

\subsection{Identifiability and the coercivity condition} \label{sec:inference}
Consider the likelihood-based inference of the interaction function $\phi$ from observation data consisting of
many independent trajectories $\{\bX^{[0,T],m}\}_{m=1}^M$.
The maximum likelihood estimator (MLE) is a maximizer of the log-likelihood ratio of these trajectories over a hypothesis space $\hypspace$:
\begin{align*}
\widehat{\intkernel}_{\hypspace,M}:=\argmax{\intkernelvar\in \hypspace} \mathcal{E}^M(\intkernelvar), \quad   \text{ with } \mathcal{E}^M(\intkernelvar)  := \frac{1}{M}\sum_{m=1}^M  \mathcal{E}_{\bX^{[0,T],m}}(\intkernelvar)
\end{align*}
where $\mathcal{E}_{\bX^{[0,T],m}}(\intkernelvar)$ denotes the average log-likelihood ratio of the trajectory $\bX^{[0,T],m}$, which is given, by the Girsanov theorem (see e.g. \cite[Section 1.1.4]{Kut04} and \cite[Section 3.5]{KS98}), by
\begin{align}\label{lkhd_cts}
\mathcal{E}_{\bX^{[0,T]}}(\intkernelvar)&= - \frac{1}{2TN}\int_{0}^T \left(|\nabla \pV_{\intkernelvar}(\bX^t)|^2 dt + 2 \langle \nabla \pV_{\intkernelvar}(\bX^t), d\bX^t \rangle \right).
\end{align}
Note that $\mathcal{E}^M(\cdot)$ is a quadratic functional. When $\hypspace$ is a finite dimensional linear space, $\mathcal{E}^M(\cdot)$ becomes a quadratic function on $\hypspace$, and an estimator $\widehat{\intkernel}_{\hypspace,M}$ can be obtained by solving a least squares problem (which has multiple solutions if the matrix of the normal equations is singular.

A key assumption in the definition of MLE is the uniqueness of the maximizer of the log-likelihood ratio $\mathcal{E}^M(\cdot)$.   When $M\to\infty$, by the Law of Large Numbers,
\[
\E\mathcal{E}_{\bX^{[0,T]}}(\intkernelvar)= \lim_{M\to \infty} \mathcal{E}^M(\intkernelvar)\quad a.s.,
\]
and the uniqueness assumption leads to the identifiability in Definition \ref{def:identifiability}:  the interaction function is identifiable if it is the unique maximizer of the expectation of the log-likelihood ratio.

The coercivity condition in Definition \ref{def_coercivty} can be readily verified from data on any finite dimensional space $\hypspace$ and it provides guidance on the choice of basis functions for $\hypspace$. In fact, with the choice of an orthonormal basis for  $\hypspace \subset L^2(\wbar \rho_T)$, the coercivity constant $c_{\hypspace,T}$ provides a lower bound for the smallest eigenvalue of the normal matrix \cite{LMT19,LMT20}, thus ensuring the uniqueness of the estimator. More importantly, if the coercivity condition holds true, the estimator is proved to be consistent and converge at a rate (in $M$) equal to the minimax rate of nonparametric regression in $1$ dimension (see \cite[Theorem 5--6]{LMT19} and \cite[Theorem 3.1--3.2]{LMT20}), and the estimation errors can be, under possibly further assumptions, dimension independent (see \cite[Theorem 9]{LMT19}).

The next proposition shows that the coercivity condition is sufficient for the identifiability and it also becomes necessary when the number of particles in the system goes to infinity. This implies Theorem \ref{thm_mainAll}(a).

\begin{proposition}\label{propID_CC}
Consider the system  \eqref{eq:sys1st} on $[0,T]$ with interaction function $\phi$. Let $\wbar \rho_T$ be as defined in \eqref{eq:rhoavg}. Then, the interaction function is identifiable on a subspace $\hypspace\subset L^2(\wbar\rho_T)$ if and only if
  \begin{align}\label{eq:identifiabitly}
\frac{1}{T}\int_{0}^T   \E[ h(|\br_{12}^t|) h(|\br_{13}^t|)\frac{\langle \br_{12}^t,\br_{13}^t \rangle}{|\br_{12}^t||\br_{13}^t|} ](t)dt
 > -  \frac{1}{N-2} \|h\|_{L^2(\wbar \rho_T)}^2, \quad \text{for all } h\neq 0\in \hypspace.
 \end{align}
Thus, it is identifiable on a linear subspace $\hypspace\subset L^2(\wbar\rho_T)$ if the coercivity condition holds on $\hypspace$. Furthermore, when $N\to\infty$, the interaction function is identifiable on $L^2(\wbar \rho_T)$ if and only if the coercivity condition holds on  $L^2(\wbar \rho_T)$.
\end{proposition}
 \begin{proof}
 Noting that $d\bX^t= - \nabla J_{\phi}(\bX^t)dt + d\bB^t$ and that $J_{\intkernelvar}$ is linear in $\intkernelvar$, we have
\begin{align*}
  & \int_0^T |\nabla \pV_{\intkernelvar}(\bX^t)|^2 dt + 2 \langle \nabla \pV_{\intkernelvar}(\bX^t), d\bX^t \rangle  \\
   = &\int_0^T |\nabla \pV_{\intkernelvar}(\bX^t)|^2 dt - 2 \langle \nabla \pV_{\intkernelvar}(\bX^t), \nabla J_{\phi}(\bX^t) \rangle dt +  \langle \nabla \pV_{\intkernelvar}(\bX^t), d\bB^t \rangle \\
  = &\int_0^T |\nabla \pV_{\intkernelvar -\phi}(\bX^t)|^2 dt -  |\nabla \pV_{\phi}(\bX^t)|^2 dt +  \langle \nabla \pV_{\intkernelvar}(\bX^t), d\bB^t \rangle,
\end{align*}
where the last inequality follows from completing the squares. Note first that from \eqref{lkhd_cts} with $\varphi= \phi$ we have $\E\mathcal{E}_{\bX^{[0,T]}}(\phi) = - \frac{1}{TN} \int_0^T  |\nabla \pV_{\phi}(\bX^t)|^2 dt$.  Then, the expectation of the negative log-likelihood ratio in \eqref{lkhd_cts}  is
\begin{align} \label{eq:Elkhd_cts}
\E\mathcal{E}_{\bX^{[0,T]}}(\intkernelvar)&=- \frac{1}{2TN}\int_{0}^T \E |\nabla \pV_{\intkernelvar -\phi}(\bX^t)|^2 dt -\E\mathcal{E}_{\bX^{[0,T]}}(\phi).
\end{align}
Thus, $\phi$ is the unique maximizer of $\E\mathcal{E}_{\bX^{[0,T]}}(\cdot)$ on $\hypspace$, i.e., is identifiable on $\hypspace$, if and only if
 \begin{equation}\label{eq:error_control}
 \frac{1}{TN}\int_{0}^T \E |\nabla \pV_{h}(\bX^t)|^2 dt >0
 \end{equation}
whenever $h=\intkernelvar -\phi\neq 0 $ in $\hypspace$.

Also, by exchangeability, with notation $\br_{ji}^t =\bX_j^t-\bX_{i}^t$,  we have
\begin{align*}
& \frac{1}{N} \E |\nabla \pV_{h}(\bX^t)|^2  = \sum_{i=1}^N\frac{1}{N^3} \sum_{\substack{ j,k=1,\\ j\neq i, k\neq i} }^N \underbrace{\E[ h(|\br_{ji}^t|) h(|\br_{ki}^t|)\frac{\langle \br_{ji}^t,\br_{ki}^t \rangle }{|\br_{ji}^t||\br_{ki}^t|} ]}_{I_{ijk}(t)}
= \frac{(N-1)[ (N-2)I_{123}+ I_{122}]}{N^2} ,
 \end{align*}
 where the equality follows from that $I_{ijk} = I_{123}$ for all triplets $\{(i,j,k), j\neq i, k\neq i, j\neq k\}$, contributing $N(N-1)(N-2)$ copies of $I_{123}$; and that $I_{ijk} = I_{122}$ for all for all triplets $\{(i,j,k), j=k\neq i\}$, contributing $N(N-1)$ copies of $I_{122}$.
 Therefore,
  \begin{align}\label{eq:ExpJh}
  \frac{1}{TN}\int_{0}^T \E |\nabla \pV_{h}(\bX^t)|^2 dt = \frac{(N-1)(N-2)}{N^2} \frac{1}{T}\int_{0}^T   I_{123}(t)dt + \frac{N-1}{N^2} \frac{1}{T}\int_{0}^T   I_{122}(t)dt.
  \end{align}
Then, Eq.\eqref{eq:identifiabitly} is equivalent to Eq.\eqref{eq:error_control} by noting that $\frac{1}{T}\int_0^T I_{122}(t) dt = \frac{1}{T}\int_0^T \E[h(|\br_{12}^t|)^2]dt = \|h\|_{L^2(\wbar \rho_T)}^2$.

 If the coercivity condition holds on $\hypspace$, i.e.  $
\frac{1}{T}\int_{0}^T   \E[ h(|\br_{12}^t|) h(|\br_{13}^t|)\frac{\langle \br_{12}^t,\br_{13}^t \rangle}{|\br_{12}^t||\br_{13}^t|} ](t)dt
 \geq c_{\hypspace,T} \|h\|_{L^2(\wbar\rho_T)}^2
$
 for all $h\in \hypspace$, so Eq.\eqref{eq:identifiabitly} holds and $\phi$ is identifiable on $\hypspace$.

 When $N\to \infty$, by \eqref{eq:Elkhd_cts}, the maximizer is unique iff (together with \eqref{eq:ExpJh})
 \begin{equation*}
 \lim_{N\to\infty} \frac{1}{TN}\int_{0}^T \E |\nabla \pV_{h}(\bX^t)|^2 dt  = \frac{1}{T}\int_{0}^T   \E[ h(|\br_{12}^t|) h(|\br_{13}^t|)\frac{\langle \br_{12}^t,\br_{13}^t \rangle}{|\br_{12}^t||\br_{13}^t|} ](t)dt >0,
 \end{equation*}
for each $h\neq 0 \in \hypspace$. Hence, for any finite dimensional linear subspace $\hypspace$, Eq.\eqref{eq:coercivityDef} holds. Thus, identifiability on $L^2(\wbar\rho_T)$ is equivalent to the coercivity condition on $L^2(\wbar\rho_T)$ when $N\to\infty$.
\end{proof}

\begin{remark}\label{rmk:compactOpterator}
When the related integral operator introduced below is compact and strictly positive (see Remark \ref{rmk_compact}), the coercivity constant  $c_{\mathcal{H},T}$ will converge to zero as the dimension of $\hypspace$ increases to infinity, because it is the smallest eigenvalue of the integral operator on  $\hypspace \subset L^2(\wbar\rho_T)$ . Thus, when performing nonparametric inference of the interaction function, even when the coercivity condition holds true on $L^2(\wbar\rho_T)$, regularization becomes necessary to control the condition number of the normal matrix when the dimension of the hypothesis space $\hypspace$ becomes large \cite{LZTM19,LMT19,LMT20}.
\end{remark}

\begin{remark}\label{rmk:differentCC}
The coercivity in Definition {\rm \ref{def_coercivty}} is sightly more restrictive than the one in the previous studies {\rm \cite{LZTM19,LMT19,LMT20}}, which is almost equivalent to the identifiability. More precisely, the definition in {\rm \cite[Definition 3.1]{LMT20}} says that the coercivity condition holds on a set $\hypspace$ if there is a positive constant $c_\hypspace$ such that
\begin{align}\label{gencoerc}
  c_{\hypspace} \|\intkernelvar\|_{L^2(\wbar\rho_T)}  \!\!\leq   \!\!\frac{1}{2NT}\int_0^T\sum_{i=1}^{N}\E | \nabla J_\varphi(\bX^t)\big|^2 dt.
 \end{align}
for all $\varphi\in \hypspace$. This is almost the identifiability in \eqref{eq:error_control}, except requiring a uniform bound $c_\hypspace$ over $\hypspace$. Definition {\rm \ref{def_coercivty} } is more restrictive than the previous one in the sense that its coercivity constant $c_{\hypspace,T}$ in \eqref{eq:coercivityDef} may be zero while the constant $c_\hypspace$ above is positive, as in \eqref{eq:relation_coerc_constnats} below.
 In fact, by \eqref{eq:ExpJh} with $h$ replaced by $\varphi$ and the fact that $\frac{1}{T}\int_0^T I_{122}(t) dt = \|\varphi\|_{L^2(\wbar \rho_T)}^2$, we can write \eqref{gencoerc} as
\begin{align*}
& c_{\hypspace} \|\intkernelvar\|_{L^2(\wbar\rho_T)}  \!\!\leq   \!\! \frac{(N-1)(N-2)}{2N^2} \frac{1}{T}\int_{0}^T   I_{123}(t)dt + \frac{N-1}{2N^2} \|h\|_{L^2(\wbar \rho_T)}^2 \\
\Longleftrightarrow \quad & ( c_{\hypspace}- \frac{N-1}{2N^2}) \|\intkernelvar\|_{L^2(\wbar\rho_T)} \leq   \frac{(N-1)(N-2)}{2N^2} \frac{1}{T}\int_{0}^T   I_{123}(t)dt,
\end{align*}
Combining with \eqref{eq:coercivityDef}, we obtain
\begin{equation}\label{eq:relation_coerc_constnats}
 c_{\hypspace}- \frac{N-1}{2N^2} \leq \frac{(N-1)(N-2)}{2N^2} c_{\hypspace, T}.
 \end{equation}
Hence, we can have $c_{\hypspace,T} \leq 0$ while having $0 < c_\hypspace \leq  \frac{(N-1)(N-2)}{2N^2} c_{\hypspace,T} + \frac{N-1}{2N^2} $.  When $N\to \infty$, we have $c_\hypspace \leq c_{\hypspace,T}$ and the two definitions are equivalent. Definition {\rm \ref{def_coercivty}} has the advantage of providing a coercivity constant independent of $N$. Also, by requiring a positive coercivity constant $c_{\hypspace,T}$ only on each finite-dimensional hypothesis space $\hypspace$, it can hold on infinite dimensional spaces such as $L^2(\wbar\rho_T)$ when $N\to\infty$, while the previous definition can not. Thus, the new definition is suitable for studying the identifiability of the interaction function in the mean-field equation.
\end{remark}

While numerically easy to verify, the above coercivity condition on a time interval is difficult to analyze directly, because it involves the average-in-time distribution $\wbar \rho_T$ which is complicated in general,
unless it is an invariant measure of the system, either when the system starts from the invariant measure, or when we consider the large time limit. The following single-time version of the coercivity condition can be analyzed directly, and involves only the single-time distribution $\rho_t$.

\begin{definition}[Coercivity condition at time $t$] \label{def_coercivty_t}
The dynamical system \eqref{eq:sys1st}  is said to satisfy the coercivity condition at time $t$ on a linear subspace $\mathcal{H}\subset L^2( \rho_t)$, where $\rho_t$ is defined in \eqref{eq:rhoavg}, if
\begin{equation} \label{eq:c_t}
c_{\mathcal{H}}(t) : = \quad \inf_{h\in  \mathcal{H}, \, \|h\|_{L^2(\rho_t)}=1} \E[h(|\br_{12}^t|)h(|\br_{13}^t|)\frac{\innerp{\br_{12}^t}{\br_{13}^t}}{|\br_{12}^t||\br_{13}^t|} ] >0,
\end{equation}
where $\br_{ij}^t = \bX^{t}_i- \bX^t_j$. If the coercivity condition holds true on every finite dimensional subspace $\mathcal{H} \subset L^2(\rho_t)$, we say the system satisfies the coercivity condition on $L^2(\rho_t)$ at time $t$.
\end{definition}

Similar to Proposition \ref{propID_CC}, if the coercivity condition holds on $\hypspace$ at $t_0$, then the interaction function can be identified on $\hypspace$ from a large size of samples at time $t_0$. This explains the observation in \cite{LZTM19, LMT19,LMT20} that the interaction function can be learned from multiple very short-time trajectories.


\subsection{Relation to strictly positive integral operators}
\label{sec:operator}
We show in this subsection that the coercivity condition on an interval $[0,T]$ (or at single-time $t$) discussed above is equivalent to the strict positivity of related integral operators on $L^2(\wbar\rho_T)$ (or $L^2(\rho_t)$).


Recall that a linear operator $Q$ on a Hilbert space $H$ is \emph{positive} if $\langle Q f,f\rangle \geq 0$ for any $f\in H$. It is said to be \emph{strictly positive} if  $\langle Q f,f\rangle> 0$ whenever $f\neq 0 \in \hypspace$. 
Hereafter, by an integral operator $Q$ with kernel $K$ on $L^2(\rho)$, we mean the bounded linear operator defined by
\[
[Qh](r) = \int K(r,s) h(s) \rho(s)ds
\]
for any $h\in L^2(\rho)$.

\begin{proposition}\label{prop.cc2PD}
The system \eqref{eq:sys1st} on $[0,T]$ satisfies the coercivity condition on $L^2( \wbar\rho_T)$  iff  the integral operator $\wbar Q_T$ on $L^2(\wbar \rho_T)$  with the integral kernel
\begin{align} \label{kernel_avg}
\widebar{K}_T(r,s): &= \frac{1}{\widebar{q}_T(r) \widebar{q}_T(s)} (rs)^{d-1}\frac{1}{T}\int_{0}^T\int_{S^{d-1}}\int_{S^{d-1}}\innerp{\xi}{\eta} p_t(r\xi,s\eta) d\xi d\eta dt
\end{align}
is strictly positive, where $ \widebar{q}_T(r)$ denotes the density of the measure $\widebar{\rho}_T$ and $p_t(u,v)$ denotes the density function of the random vector $(\br_{12}^t, \br_{13}^t)$.
\end{proposition}
\begin{proof} By definition, we have
\begin{align}\label{eq:Qavg}
[\wbar Q_T h] (r)= \int \wbar{K}_T(r,s) h(s) \wbar \rho_T(s)ds.
\end{align}
Note first that  for any $h,g \in L^2(\wbar \rho_T)$, by a change of variable to the polar coordinates, we have
\begin{align}
&\frac{1}{T}\int_{0}^T  \E[h(|\br_{12}^t|)g(|\br_{13}^t|)\frac{\innerp{\br_{12}^t}{\br_{13}^t}}{|\br_{12}^t||\br_{13}^t|} ]dt
= \frac{1}{T}\int_{0}^T \int_{\R^d}\int_{\R^d} h(u)g(v) \frac{\innerp{u}{v}}{|u||v|} p_t(u,v)dudv dt \notag \\
 =& \frac{1}{T}\int_{0}^T \int_{\R^+} \int_{\R^+}   \int_{S^{d-1}}\int_{S^{d-1}} h(r)g(s) p(r\xi, s\eta) \innerp{\xi}{\eta}(rs)^{d-1}   d\xi d\eta drds dt \notag \\
 = & \int_{\R^+} \int_{\R^+} \widebar{K}_T(r,s) h(r) g(s)\widebar{\rho}_T(dr) \widebar{\rho}_T(ds) = \langle \wbar Q_T h, g\rangle_{L^2(\wbar \rho_T)}. \label{eq:Q&Exp}
\end{align}

Then, to show the equivalence between the strictly positive-definiteness of $\wbar Q_T$ and the coercivity condition on $L^2(\wbar\rho_T)$, it suffices to note that by the above equality, the coercivity constant in \eqref{eq:coercivityDef} satisfies
\[
c_{\hypspace, T} = \inf_{h\in  \mathcal{H}, \, \|h\|_{L^2(\wbar \rho_T)}=1} \langle \wbar Q_T h, h\rangle_{L^2(\wbar \rho_T)} >0,
\]
for any finite dimensional linear subspace $\mathcal{H}  \subset L^2(\wbar \rho_T)$.
\end{proof}

\begin{remark} \label{rmk_compact}
It follows from \eqref{eq:Q&Exp} that $\wbar Q_T$ is a symmetric bounded linear operator on $L^2(\wbar \rho_T)$:
\begin{align*}
\langle \wbar Q_T h, g\rangle_{L^2(\wbar \rho_T)} 
 & \leq  \frac{1}{T}\int_{0}^T \E[h(|\br_{12}^t|)g(|\br_{13}^t|)] dt
 \leq \|h\|_{L^2(\wbar \rho_T)} \|g\|_{L^2(\wbar \rho_T)},
\end{align*}
where the second inequality follows from the Cauchy-Schwarz inequality and the fact that $\frac{1}{T}\int_0^T \E[h(|\br_{12}^t|)^2]dt = \|h\|_{L^2(\wbar \rho_T)}^2$. When the integral kernel $\wbar K_T$ is in $L^2(\wbar \rho_T\otimes \wbar \rho_T)$, the operator $\wbar Q_T$ is compact and its eigen-functions  form an orthonormal basis of $L^2(\wbar \rho_T)$. 
Then, for any linear subspace $\hypspace\in L^2(\wbar\rho_T)$, the coercivity constant $c_{\hypspace,T}$ is the smallest eigenvalue of $\wbar Q_T$ on $\hypspace$.
\end{remark}

Similarly, we have the following proposition for the coercivity condition at a single time.
\begin{proposition}\label{prop.cc2PD_t}
The system \eqref{eq:sys1st} satisfies the coercivity conditions on $L^2(\rho_t)$ at time $t$ iff the integral operator $Q_t$ with  kernel
 \begin{equation} \label{eq:kernelK_t}
K_t(r, s) := \frac{1}{q_t(r) q_t(s)} (rs)^{d-1}\int_{S^{d-1}}\int_{S^{d-1}}\innerp{\xi}{\eta} p_t(r\xi,s\eta) d\xi d\eta.
 \end{equation}
 is strictly positive on $L^2(\rho_t)$, where $q_t$ denotes the density of the measure $\rho_t$ and $p_t(u,v)$ denotes the density function of the random vector $(\br_{12}^t, \br_{13}^t)$.
\end{proposition}
\begin{proof}
Note that
\begin{align}\label{Q_tExp}
\E\left[h(|\br_{12}^t|)h(|\br_{13}^t|)\frac{\innerp{\br_{12}^t}{\br_{13}^t}}{|\br_{12}^t||\br_{13}^t|}\right]  =
\int_0^\infty \int_0^\infty h(r)h(s)  K_t(r,s)\rho_t(dr) \rho_t(ds) = \langle Q_t h, h\rangle_{L^2(\rho_t)}.
\end{align}
Then, $c_\hypspace(t) = \inf_{h\in  \mathcal{H}, \, \|h\|_{L^2(\rho_t)}=1} \langle Q_t h, h\rangle_{L^2(\rho_t)}$. Hence, $Q_t$ is strictly positive iff $c_\hypspace(t) >0$ for any finite dimensional linear subspace $\mathcal{H}  \subset L^2(\rho_t)$.
 \end{proof}




\bigskip

The positivity of $Q_t$ on $L^2(\rho_t)$ will be proved in later sections by showing that its integral kernel $K_t$ is strictly positive-definite, using the special structure of $p_t(u,v)$ (the density function of the random vector $(\br_{12}^t, \br_{13}^t)$). However, the positivity of $\wbar Q_T$ on on $L^2(\wbar \rho_T)$ can not be proved in the same way because the kernel $\widebar{K}_T$ in \eqref{kernel_avg} involves the average of $p_t(u,v)$.
The following proposition shows that, while $Q_t$ and $\wbar Q_T$ are defined on different spaces (unless $\rho_t$ is the same as $\wbar\rho_T$), the strict positivity of $Q_t$ for all $t\in [0,T]$ implies the positivity of $\wbar Q_T$.

\begin{proposition}\label{prop:Qavg_Qt}
The integral operator $\wbar Q_T$ on $L^2(\wbar\rho_T)$ with  kernel $\widebar{K}_T$ in  \eqref{kernel_avg}  is strictly positive on $\hypspace$
if the family of operators $ Q_t$ on $L^2(\rho_t)$  with kernel $K_t$ in \eqref{eq:kernelK_t} is strictly positive for each $t\in [0,T]$.
\end{proposition}
\begin{proof} Note that a combination of \eqref{eq:Q&Exp}  and \eqref{Q_tExp} leads to
\begin{equation}\label{eq:QT_Qt}
\langle \wbar Q_T  h, h\rangle_{L^2(\wbar\rho_T)}  = \frac{1}{T} \int_0^T \langle Q_t h, h\rangle_{L^2(\rho_t)} dt.
\end{equation}
For any $ h\neq 0$, we have $ \langle Q_t h, h\rangle_{L^2(\rho_t)}>0$ for all $t$ because $Q_t$ is strictly positive. Thus, $\langle \wbar Q_T  h, h\rangle_{L^2(\wbar\rho_T)} >0$. This implies that $\wbar Q_T$ is strictly positive on $L^2(\wbar\rho_T)$.
\end{proof}

\begin{remark}
If we let $\hypspace\subseteq C_b(\R^+)$, then $\wbar Q_T$ is strictly positive on $\hypspace$ if $ Q_t$  is non-negative for each $t\in [0,T]$ and strictly positive on $\hypspace$ for some $t_0\in [0,T]$, because the integrand $ g(t) = \langle Q_t h, h\rangle_{L^2(\rho_t)}$ in {\rm \eqref{eq:QT_Qt}} is continuous in time $t$ for each $h\in \hypspace$. Indeed, note that $\E[f(\bX^t)]$ is continuous in $t$ for any $f\in C_b(\R^{Nd},\R)$ because $\bX^t$ is a diffusion process whose density satisfies the Kolmogorov forward equation. Thus, for any function $h\in C_b{\R^+}$, the function  $f(\bX^t):=h(|\br_{12}^t|)h(|\br_{13}^t|)\frac{\innerp{\br_{12}^t}{\br_{13}^t}}{|\br_{12}^t||\br_{13}^t|}$ is continuous and bounded, so $ \langle Q_t h, h\rangle_{L^2(\rho_t)} = \E[f(\bX^t)] $ in \eqref{Q_tExp} is continuous in $t\in [0,T]$.
\end{remark}

\section{The case of linear systems} \label{sec:OU}
We prove in this section that the coercivity condition holds true for linear systems with general non-degenerate exchangeable Gaussian, covering Theorem \ref{thm_mainAll}(b). We start with a macro-micro decomposition (with the average position of all particles $\bX_c^t$ being the macro-state and with $\bY^t= \bX^t-\bX_c^t$ being the micro-state) to transform the system into a system of decentralized positions, which is ergodic. Then, we prove the coercivity condition when the initial distribution has a special structure similar to the covariance of the invariant measure (see Theorem \ref{thm_main_OU}), by using a series representation of the integral kernel and by a M\"untz-type theorem about polynomials with even degrees being dense in $L^2(\R^+,\mu)$ for a proper $\mu$.  Lastly, we prove the coercivity condition for general initial distributions by extending the arguments to non-radial interaction kernels (see Theorem \ref{thm:OU_non-radial}).

 \subsection{A macro-micro decomposition}
We first consider linear systems for which we have $\phi(r) = \thetaOU r$ with $\thetaOU>0$ (i.e., $\Phi(r)= \frac{1}{2}\thetaOU r^2$, a quadratic potential). The system \eqref{eq:sys1st} can be written as
\begin{align} \label{eq:OU}
 d\bX^t&= -\thetaOU \btheta \bX^tdt +   d\bB^t,
 \end{align}
 where the matrix $\btheta\in \R^{Nd\times Nd}$ is given by (with $I_d$ being the identity matrix on $\R^d$)
 \begin{equation} \label{matA}
\btheta = \frac{1}{N}
 \begin{pmatrix}
  (N-1)I_d& -I_d & \cdots & -I_d  \\
-I_d  & (N-1)I_d & \cdots &-I_d \\
  \vdots  & \vdots  & \ddots & \vdots  \\
-I_d  & -I_d & \cdots & (N-1)I_d
 \end{pmatrix}.
 \end{equation}
It is straightforward to see that $\btheta^2= \btheta$, and that the matrix $\btheta$ has eigenvalue $1$ of multiplicity $(N-1)d$ and eigenvalue $0$ of multiplicity $d$ (with a null space $\{\bx=c (\mathbf{v,v,\cdots, v}): c\in \R, \mathbf{v}\in \R^d\}$).

By a macro-micro decomposition of the system as in   \cite{malrieu2003convergence,cattiaux2018stochastic},
the next lemma shows that the center of the particles moves like a Brownian motion, and the particles concentrate around the center with a distribution close to Gaussian.
\begin{lemma}\label{leamm:Linear}
(i)
The solution $\bX^t$ of Eq.\eqref{eq:OU} can be explicitly written as
\begin{equation}\label{eq:Xt_int_Xc}
\bX^t = e^{-\thetaOU t} \btheta \bX^0 +  \int_0^t  e^{-\thetaOU (t-s)}\btheta d\bB^s + \bX_c^t,
\end{equation}
where $\bX_c^t= (\mathbf{v^t,v^t,\cdots, v^t})'$ with $\mathbf{v^t}:=\frac{1}{N}\sum_{i=1}^N \bX_i^t= \frac{1}{N}\sum_{i=1}^N (\bX_i^0+\bB_i^t)$.\\
(ii) Conditional on $\bX^0$, the centralized process
\[ \bY^t= \bX^t- \bX_c^t
\]
 is an Ornstein-Uhlenbeck process with distribution $\mathcal{N}\left(e^{-\thetaOU t} \btheta\bX^0, \frac{1}{2\thetaOU}(1-e^{-\thetaOU t})\btheta \right)$ for each $t$. In particular, if $\bX^0$ is Gaussian with variance $\mbf\Sigma$, then for each $t$, $\bY^t$ has a  distribution $\mathcal{N}\left(\mathbf{0},e^{-2\theta t}\btheta\mbf\Sigma \btheta+ \frac{1}{2\thetaOU}(1-e^{-\thetaOU t})\btheta \right)$.
\end{lemma}
\begin{proof}
The fact that $\mathbf{v^t}:=\frac{1}{N}\sum_{i=1}^N \bX_i^t = \frac{1}{N}\sum_{i=1}^N (\bX_i^0+\bB_i^t)$ follows directly from the equation
\[d\mathbf{v^t} = \frac{1}{N}\sum_{i=1}^N d\bX_i^t = \frac{1}{N}\sum_{i=1}^N d\bB_i^t.\]
Next, note that $\bY^t= \bX^t- \bX_c^t = \btheta\bX^t$ and
\[
d\bY^t = \btheta d\bX^t = -\thetaOU \btheta^2 \bX^tdt +  \btheta d\bB^t = -\thetaOU \bY^tdt +  \btheta d\bB^t , \]
where we used $\btheta^2 =\btheta $ in the third equality.
Therefore, $(\bY^t)$ is an Ornstein-Uhlenbeck process
\[
\bY^t = e^{-\thetaOU t}\bY^0 +  \int_0^t e^{-\thetaOU (t-s)}\btheta d\bB^s.
\]
Therefore, conditional on $\bX^0$, with $\bY^0=\btheta \bX^0$ and $\btheta^2=\btheta$, we have that  the distribution of $\bY^t$ is $\mathcal{N}\left(e^{-\thetaOU t} \btheta\bX^0, \frac{1}{2\thetaOU}(1-e^{-\thetaOU t})\btheta \right)$ and that $\bX^t= \bX_c^t+ \bY^t$ can be written as in \eqref{eq:Xt_int_Xc}.

 If the initial distribution $\bX^0$ is exchangeable, then $\E[\bY^0]=\btheta\E[\bX^0]=\mathbf{0}$ because $\E[\bX_i^0]=\E[\bX_j^0]$ for any $(i,j)$. Thus, if  $\bX^0$ is Gaussian and exchangeable, then $\bY^t$ is Gaussian with mean $\mathbf{0}$. The variance of $\bY^t$ follows directly from the above integral representation.
\end{proof}


\subsection{Coercivity condition for linear systems}
We begin with a technical lemma for generic Gaussian random vectors. We denote by ${\rm cov}(X,Y)$ the covariance of $X$ and $Y$, with the convention that ${\rm cov}(X)= {\rm cov}(X,X)$.

\begin{lemma}\label{SP_Gaussian}
Let $(X, Y, Z)$ be a Gaussian vector in $ \R^{3d}$ with an exchangeable  joint distribution and $\cov(X)-\cov(X,Y) =\lambdaG I_d$ for some $\lambdaG>0$. Denote $ p^\lambdaG(u,v)$ the joint distribution of $(X-Y, X-Z)$ and denote $\rho^\lambdaG$ the distribution (and $q^\lambdaG(r)$ the density function) of $|X-Y|$.
 Then
 \begin{itemize}
  \item[(i)] The function $K^\lambdaG(r,s):\R^+\times \R^+\to\R$ defined by
  \begin{equation} \label{kernelK_gauss}
K^\lambdaG(r,s) := \frac{1}{q^\lambdaG(r)q^\lambdaG(s)} (rs)^{d-1}\int_{S^{d-1}}\int_{S^{d-1}}\innerp{\xi}{\eta} p^\lambdaG(r\xi,s\eta) d\xi d\eta
 \end{equation}
 is  uniformly bounded and is in $L^2(\rho^\lambdaG\otimes\rho^\lambdaG)$.
\item[(ii)] The integral operator $Q^\lambdaG$ with kernel $K^\lambdaG $ is strictly positive on $L^2(\rho^\lambdaG)$, i.e., for any $0\neq h\in L^2(\rho^\lambdaG)$,
\begin{equation}  \label{eq:Q_expSPD}
\langle Q^\lambdaG h, h\rangle_{L^2(\rho^\lambdaG)} =   \E\left[h(|X-Y|)h(|X-Z|)\frac{\langle X-Y,X-Z\rangle}{|X-Y||X-Z|}\right] >0.
\end{equation}
 \end{itemize}
\end{lemma}

\begin{proof}
We start with explicit expressions for  $p^\lambdaG(u,v)$, $\rho^\lambdaG(r)$ and $K^\lambdaG(r,s)$.
By exchangeability, the random vector $(X-Y, X-Z)$ is centered Gaussian with covariance matrix $\lambdaG\begin{pmatrix}
 2  I_d      & I_d\\
I_d         & 2I_d
 \end{pmatrix}$, whose inverse is $\frac{1}{3\lambdaG}\begin{pmatrix}
 2I_d    & -I_d\\
   -I_d   & 2I_d
 \end{pmatrix}$. Thus, the joint distribution is
 \[ p^\lambdaG(u,v) = (2\sqrt{3}\pi \lambdaG)^{-d} e^{-c_\lambdaG (|u|^2+|v|^2- \inp{u,v}) }, \quad \text{with } c_\lambdaG = \frac{1}{3\lambdaG},
 \]
 and a direct computation yields that the density of $|X-Y|$ is
 \[
 q^\lambdaG(r)= \frac{1}{C_\lambdaG} r^{d-1}e^{-\frac{r^2}{4\lambdaG}} \mathbf{1}_{r\geq 0}, \quad \text{ with } C_\lambdaG= \frac{1}{2} (4\lambdaG)^{\frac{d}{2}}\Gamma(\frac{d}{2}),
 \]
where the constant $\Gamma(\frac{d}{2})$ comes from the surface area of the unit sphere in $\R^d$:  $|S^{d-1}|= \frac{2\pi^{d/2}}{\Gamma(\frac{d}{2})}$.

  Then,  the integral kernel in \eqref{kernelK_gauss} can be written as
\begin{align*}
 K^\lambdaG(r,s) &= C_d e^{-\frac{1}{12\lambda}(r^2+s^2)}\int_{S^{d-1}}\int_{S^{d-1}}\inp{\xi,\eta} e^{c_\lambda rs\inp{\xi,\eta}} \frac{d\xi d\eta}{|S^{d-1}|^2}, \quad \text{with } C_d= (\frac{\sqrt{3}}{2})^{-d}.
 \end{align*}
 Here when $d=1$, the above spherical measure on $S^0= \{-1,1\}$ is interpreted as $\mathbb{P}(\xi=1) =\mathbb{P}(\xi=-1) = \frac{1}{2}$, or equivalently, $\int_{S^{d-1}}\int_{S^{d-1}}\inp{\xi,\eta} e^{\frac{1}{3\lambdaG}(rs\inp{\xi,\eta})} \frac{d\xi d\eta}{|S^{d-1}|^2} = \frac{1}{2}(e^{c_\lambda rs} -e^{-c_\lambda rs})$.

To prove (i), note that $\inp{\xi,\eta}  e^{c_\lambdaG rs\inp{\xi,\eta}} \leq e^{c_\lambda rs}$. Then,
\[K^\lambdaG(r,s) \leq C_d e^{-\frac{1}{12\lambda}(r^2+s^2) + c_\lambda rs},\]
 and it follows that $K^\lambdaG(r,s)$ is uniformly bounded above and is in $L^2(\rho^\lambdaG\otimes\rho^\lambdaG)$.

 To prove (ii), we first represent $K^\lambdaG(r,s)$ in terms a series of polynomials and then apply a M\"untz-type theorem.
Note that by  Taylor expansion,
$
\inp{\xi,\eta}  e^{c_\lambdaG rs\inp{\xi,\eta}}
=\sum_{k=1}^\infty \frac{1}{(k-1)!} c_\lambdaG^{k-1} (rs)^{k-1}  \inp{\xi,\eta}^k,
$ and that
\[
b_k=\int_{S^{d-1}}\int_{S^{d-1}} \inp{\xi,\eta}^k  \frac{d\xi d\eta}{|S^{d-1}|^2} \left\{
                \begin{array}{ll}
                  = 0, & \textrm{ for odd } k, \\
                  \in (0,1),   & \textrm{ for even } k,
                \end{array}
              \right.
 \]
 due to symmetry. We have
\begin{align*}
 K^\lambdaG(r,s)= C_d e^{-\frac{1}{12\lambda}(r^2+s^2)} \sum_{k=0}^\infty \frac{1}{k!} c_\lambdaG^k b_{k+1} (rs)^{k} =C_d e^{-\frac{1}{12\lambda}(r^2+s^2)} \sum_{k=1,k \text{ odd}}^\infty \frac{1}{k!} c_\lambdaG^k b_{k+1} (rs)^{k-1}.
  \end{align*}
Then, for any $h\in L^2(\rho^\lambdaG)$,  we have
\begin{align*}
 \langle Q^\lambdaG h, h\rangle_{L^2(\rho^\lambdaG)}  & = \int_0^\infty h(r)h(s)K^\lambdaG(r,s)  \rho^\lambdaG(r)\rho^\lambdaG(s)drds \\
& = C_d \sum_{k=1, k\text{ odd} }^\infty \frac{1}{k!} c_\lambdaG^k b_{k+1} \left(\int_0^\infty  h(r) r^{k-1} e^{-\frac{1}{12\lambda}r^2} \rho^\lambdaG(r)dr \right)^2\geq 0.
 \end{align*}
 Note that
 \[ \int_0^\infty  h(r) r^{k-1} e^{-\frac{1}{12\lambda}r^2} \rho^\lambdaG(r)dr  = C_\lambdaG^{-1} \int_0^\infty  h(r) r^{k+d-2}e^{-\frac{1}{3\lambda}r^2} dr.  \]
By Lemma \ref{lemma:Muntz}, a variation of the M\"untz Theorem,  the space ${\rm span}\{1,r^2, r^4, r^6,\cdots\}$ is dense in $L^2(r^{d-1}e^{-\frac{1}{3\lambda}r^2})$.
Thus, $\langle Q^\lambdaG h, h\rangle_{L^2(\rho^\lambdaG)}=0$ only if $h \equiv 0$, and $Q^\lambdaG$ is strictly positive.  Eq.~\eqref{eq:Q_expSPD} follows as in Eq.\eqref{eq:Q&Exp}.
\end{proof}

\bigskip
The next theorem implies Theorem \ref{thm_mainAll}(b) under the additional assumption that the initial distribution of $\bX^0$ is exchangeable Gaussian with a covariance satisfying $\cov(\bX_i^0)-\cov(\bX_i^0,\bX_j^0)=\lambda_0 I_d$ for any $1\leq i<j\leq N$.
Intuitively, we may decompose each component of $\bX^0$ as the sum of a common variable and an independent variable, i.e., $\bX_i^0 = Z_i+ W$, where $\{Z_i\}_{i=1}^N$ are i.i.d. $\mathcal{N}(0,\lambda_0 I_d)$ and $W$ is a common Gaussian random variable, and this implies that the  particles are initially scattered randomly around a random position.

\begin{theorem} \label{thm_main_OU}
Suppose the linear system \eqref{eq:OU} starts from $\bX^0$ whose distribution is exchangeable Gaussian with covariance satisfying $\cov(\bX_1^0)-\cov(\bX_1^0,\bX_2^0)=\lambda_0 I_d$ for some $\lambda_0>0$. Then, $\mathrm{(i)}$ the coercivity condition holds  on $L^2(\rho_t)$ at each time $t\in [0,T]$ as in Definition {\rm \ref{def_coercivty_t}};
$\mathrm{(ii)}$ the coercivity condition holds on $L^2(\wbar \rho_T)$ on $[0,T]$ as in Definition \rm{\ref{def_coercivty}}.
\end{theorem}
\begin{proof} Let $\bY^t= \bX^t- \bX_c^t$. Note that
\[\br_{ij}^t = \bX^{t}_i- \bX^t_j =\bY^{t}_i- \bY^t_j.\]
Thus,  the coercivity conditions for the process $(\bX^t)$ is equivalent to those for the process $(\bY^t)$.

With $M:=\cov(\bX_i^0,\bX_j^0)$, the covariance of $\bX^0$ can be written as
\[\mbf\Sigma =
 \begin{pmatrix}
  M + \lambda_0 I_d& M & \cdots & M  \\
M & M + \lambda_0 I_d & \cdots & M \\
  \vdots  & \vdots  & \ddots & \vdots  \\
M & M & \cdots & M + \lambda_0 I_d
 \end{pmatrix}.
\]
A direct computation shows that $\btheta\mbf\Sigma \btheta = \lambda_0 \btheta$ for $\btheta$ defined in \eqref{matA}. By Lemma \ref{leamm:Linear}(ii), this special covariance implies that the centralized process $\bY^t= \bX^t- \bX_c^t$ has a covariance $\lambda(t)\btheta$ with $\lambda(t)=\left[e^{-2\theta t}\lambda_0+\frac{1}{2\theta}(1-e^{-\theta t})\right]$.
Then, $(\bY_1^t,\ldots,\bY_N^t)$ is exchangeable Gaussian with covariance satisfying $\cov(\bY_i^t)-\cov(\bY_i^t,\bY_j^t)=\lambda(t) I_d$, particularly for the vector $(\bY^t_1,\bY^t_2,\bY^t_3)$.  Then, applying Lemma \ref{SP_Gaussian} to the vector $(\bY^t_1,\bY^t_2,\bY^t_3)$, we obtain that the integral operator $Q^{\lambda(t)}$ with kernel $K^{\lambda(t)}$ defined in \eqref{kernelK_gauss} is strictly positive on $L^2(\rho_t)$ for each $t$.  In the notation of Proposition  \ref{prop.cc2PD_t}, this implies that integral operator $Q_t= Q^{\lambda(t)}$ with kernel  $K_t= K^{\lambda(t)}$ is strictly positive on $L^2(\rho_t)$. Part (i) then follows.

Since $Q_t$ is strictly positive on $L^2(\rho_t)$ for each $t\in [0,T]$,  so is $\wbar Q_T$ on $L^2(\wbar \rho_T)$ by Proposition \ref{prop:Qavg_Qt}. Then, the coercivity condition holds on $L^2(\wbar \rho_T)$ by Proposition \ref{prop.cc2PD}.
 \end{proof}

\begin{remark}
When the system is deterministic, i.e. there is no stochastic force, the coercivity conditions hold true on $L^2(\wbar \rho_T)$ when the initial distribution is exchangeable Gaussian with $\cov(\bX_i^0)-\cov(\bX_i^0,\bX_j^0)=\lambda_0 I_d$.  In fact, we have $\bX^t = e^{-\thetaOU t}\btheta \bX^0+\bX_c^t$ and $\bY^t = e^{-\thetaOU t}\btheta\bX^0 $.
Then the vector $(\bY_1^t,\bY_2^t,\bY_3^t)$ is exchangeable Gaussian with $\cov(\bY_i^t)-\cov(\bY_i^t,\bY_j^t)=e^{-2\theta t}\lambda_0 I_d$. The  coercivity condition follows again from Lemmas {\rm \ref{SP_Gaussian}} and Proposition {\rm \ref{prop.cc2PD}}. In particular, it holds when the initial distribution is standard Gaussian, i.e., the initial position of the particles are i.i.d. Gaussian.
\end{remark}

\subsection{Coercivity condition for non-radial interaction functions}
The covariance constraint $\cov(\bX_i^0)-\cov(\bX_i^0,\bX_j^0)=\lambda_0 I_d$ in Theorem \ref{thm_main_OU} is necessary for the above proof, due to the need of a series representation of the radial integral kernel $K_t$. Here we remove this constraint by using a series representation of the corresponding non-radial integral kernel. More importantly, we show that the coercivity condition holds true on $L^2$ for interaction functions that are non-radial.

More precisely, consider the system with a non-radial interaction kernel $\phi:\R^d\to \R$,
\begin{equation}\label{eq:sys1st_nonradial}
  d{\bX_i^t} =\frac{1}{N} \sum_{1\leq j\leq N, j\neq i} \phi(\bX_{j}^t - \bX_i^t) \frac{\bX_{j}^t - \bX_i^t}{|\bX_{j}^t - \bX_i^t|}dt+ \sigma d\bB_i^t, \quad \text{for $i = 1, \ldots, N$},
\end{equation}
with initial condition $\bX^0$. We will extend the above coercivity condition to non-radial interaction functions.

It is straightforward to see from Section \ref{sec:inference} that for non-radial interaction functions, the function space of learning is $L^2(\R^d,\wbar{\rho}_T )$ or $L^2(\R^d,\rho_t )$ with
 \begin{equation} \label{eq:rhoavg_Rd}
  \wbar{\rho}_T(du) = \frac{1}{T} \int_0^T \rho_t(du) dt, \quad\text{ with } \rho_t(du)=  \E[\delta(\bX_i^t-\bX_{j}^t\in du)].
 \end{equation}
 Correspondingly, the coercivity condition is on $L^2(\R^d,\wbar{\rho}_T)$ or $L^2(\R^d,\rho_t )$. For simplicity of notation, we denote them by $L^2(\wbar \rho_T)$ and $L^2(\rho_t)$.

\begin{definition}[Coercivity condition for non-radial functions] \label{def_coercivty_Rd}
The 
dynamical system \eqref{eq:sys1st_nonradial} on $[0,T]$ with an initial condition $\bX^0$ and an interaction function $\phi:\R^d\to\R$ is said to satisfy the coercivity condition on a finite dimensional linear subspace $\mathcal{H}\subset L^2(\wbar \rho_T)$, with $\bar \rho_T$ defined in \eqref{eq:rhoavg_Rd},
if
\begin{equation} \label{eq:coercivityDef_Rd}
c_{\mathcal{H},T} : = \quad \inf_{h\in  \mathcal{H}, \, \|h\|_{L^2(\wbar \rho_T)}=1}\frac{1}{T}\int_{0}^T \E[h(\br_{12}^t)h(\br_{13}^t)\frac{\innerp{\br_{12}^t}{\br_{13}^t}}{|\br_{12}^t||\br_{13}^t|} ]dt >0,
\end{equation}
where $\br_{ij}^t = \bX^{t}_i- \bX^t_j$. If the coercivity condition holds true on every finite dimensional linear subspace $\mathcal{H} \subset L^2(\wbar \rho_T)$, we say the system satisfies the coercivity condition. Similarly, we can define the coercivity condition at a single time $t$ on $L^2(\rho_t)$.
\end{definition}

The next theorem shows that the coercivity condition holds true for the linear system when the distribution of $\bX^0$ is non-degenerate exchangeable Gaussian.
\begin{theorem} \label{thm:OU_non-radial}
Suppose the linear system \eqref{eq:OU} 
starts with an initial condition $(\bX_1^0,\ldots,\bX_N^0)$ whose distribution is non-degenerate exchangeable Gaussian. Then the coercivity condition holds true at each time $t\geq 0$, as well as on $[0,T]$, in the sense of Definition {\rm \ref{def_coercivty_Rd}}.
\end{theorem}

As in the previous section, we prove the coercivity condition by showing that the corresponding integral operator is strictly positive, based on a series representation of the non-radial integral kernel. Propositions \ref{prop.cc2PD} and \ref{prop.cc2PD_t} can be directly generalized to a non-radial version. We begin with the following lemma, which is a counterpart of the M\"untz-type theorem, showing that polynomials are dense in weighted $L^2$ spaces.
\begin{lemma}\label{poly} {\rm \cite[Lemma 1.1]{bs92}} Let $\mu$ be a measure on $\mathbb{R}^d$ satisfying
\[\int e^{c|x|}d\mu(x)<\infty\]
for some $c>0$, where $|x|=\sum_{j=1}^d|x_j|$. Then the polynomials are dense in $L^2(\mu)$.
\end{lemma}
\begin{proposition}\label{Gaussian}
Let $(X, Y, Z)  $ be a Gaussian vector in $ \R^{3d}$ with an exchangeable and non-degenerate joint distribution.  Denote by $ p(u,v)$  the non-degenerate joint density of $(X-Y, X-Z)$ and by $\rho$ the distribution (with a density $q$) of $X-Y$.
 Then, the integral operator $Q$ with kernel
  \begin{equation} \label{kernelK_gauss2}
K(u,v) := \frac{1}{q(u)q(v)} \frac{\innerp{u}{v}}{|u||v|} p(u,v)
 \end{equation}
is strictly positive on $L^2(\rho)$.
\end{proposition}

\begin{proof} For any $h,g\in L^2(\rho)$, from the definition of $Q$ and the Cauchy-Schwarz inequality, we have
\[|\innerp{Qh}{g}_ {L^2(\rho)}|=\Big|\E[h(X-Z)g(X-Z)\frac{\innerp{X-Y}{X-Z}}{|X-Y||X-Z|}]\Big|\le \|h\|_{L^2(\rho)}\|g\|_{L^2(\rho)}.\]
Thus, $Q$ is  a bounded operator on $L^2(\rho)$. We show that $Q$ is strictly positive by
\begin{itemize}
\item [(i)] $\innerp{Qh}{h}\ge0$ for any $h\in L^2(\rho)$, and
\item [(ii)] $\innerp{Qh}{h}=0\Rightarrow h=0$ in $L^2(\rho)$.
\end{itemize}

To prove (i),  note that
\begin{equation}\label{Qhh}
  \innerp{Qh}{h}_ {L^2(\rho)}= \iint h(u)h(v) K(u,v) \rho(u)\rho(v) dudv= \iint h(u)h(v)\frac{\innerp{u}{v}}{|u||v|}p(u,v)dudv.
\end{equation}
 Then, by Theorem \ref{t52}, it suffices to show $p(u,v)$ is positive-definite.

Since the joint distribution of $(X,Y,Z)$ is exchangeable non-degenerate Gaussian, there exist a vector $\mu \in \R^d$ and an invertible matrix $\Sigma$ such that the distribution of each of $X,Y,Z$ is $N(\mu, \Sigma)$. Decompose $\Sigma^{-1}=LL^T$, and let $\tilde X = L(X-\mu), \tilde Y = L(Y-\mu), \tilde Z = L(Z-\mu) $. Then, the distribution of each of $\tilde X,\tilde Y,\tilde Z$ is $N(0,I_d)$ and their joint distribution is exchangeable and non-degenerate. It follows that $ \cov(\tilde X,\tilde Y)=\cov(\tilde X,\tilde Z)=\cov(\tilde Y,\tilde Z)$.  Let $\Theta=\cov(\tilde X,\tilde Y) =\E[\tilde X\tilde Y^T]$. Since $\Theta$ is real symmetric, it is diagonalizable and there is a real orthonormal matrix $P$ such that $P\Theta P^T={\rm diag}(\lambda_1,...,\lambda_d)$, where $\{\lambda_1,...,\lambda_d\}$ are eigenvalues of $\Theta$. Note that $-1\le\lambda_i<1$ because the joint distribution of $(\tilde X,\tilde Y,\tilde Z)$ is non-degenerate and $\cov(\tilde X) =\cov(\tilde Y) =I_d$.

Let $X'= P\tilde X = LP(X-\mu), Y' = P\tilde Y = LP(Y-\mu), Z' = P\tilde Z = LP(Z-\mu)$. By Theorem \ref{t52}, to  prove that $p(u,v)$ is positive-definite, it suffices to show that the density $q(u,v)$ of $(X'-Y', X'-Z')$ is positive-definite.
Note that the covariance matrix $\cov(X'-Y',X'-Z')$ is invertible:
\[\cov(X'-Y',X'-Z')^{-1}=\begin{bmatrix}2I-2\Theta& I-\Theta\\I-\Theta & 2I-2\Theta\end{bmatrix}^{-1} = \frac13\begin{bmatrix}2\Lambda & -\Lambda\\
-\Lambda& 2\Lambda\end{bmatrix},\]
where $\Lambda:={\rm{diag}}(a_1,...,a_d)$ with $a_i= \frac1{1-\lambda_i}\geq \tfrac12$.
Thus, with a normalizing constant $C_d>0$, we have
\begin{align*}
q(u,v) & =  C_d\exp\left(-\tfrac12(u,v)\cov(X'-Y',X'-Z')^{-1}(u,v)^T\right) \\
& =C_d\exp\left( -\tfrac13\sum_{i=1}^d a_i ( u_i^2 +  v_i^2) +\tfrac13 \sum_{i=1}^d a_i u_iv_i \right).
\end{align*}
 By Theorem \ref{t52},  $q(u,v)$ is positive-definite.

\bigskip
To prove  (ii), let $h\in L^2(\rho)$ satisfy $\innerp{Qh}{h}=0$. We need to prove that $h=0$. Let
\begin{align*}
g_j(u):&=\exp\left(-\frac1{12}\sum_{i=1}^d a_i u_i^2 \right) h((PL)^{-1}u)\frac{((PL)^{-1}u)_j}{|(PL)^{-1}u|} \\
q(u):&=\int q(u,v)dv =C_d'\exp\left(-\frac14\sum_{i=1}^da_iu_i^2\right) \\
f(u,v):&=\exp\left(\frac13\sum_{i=1}^d a_i u_iv_i\right).
\end{align*}
By the linear transform $X\mapsto PL(X-\mu)$, we can rewrite \eqref{Qhh} as
\[\innerp{Qh}{h}=C_d(C_d')^{-2}(\mathrm{det}\ L)^2\sum_{j=1}^d\iint  g_j(u)q(u) g_j(v)q(v)f(u,v)dudv.\]
Note that the Taylor series of  $f(u,v)$ is
\[f(u,v)=\sum_{k=0}^\infty\frac1{k!3^k}(\sum_{i=1}^da_iu_iv_i)^k=\sum_{k=0}^\infty\sum_{i_1+...+i_d=k}C_{k,i_1,...,i_d} (u_1v_1)^{i_1}\cdots(u_dv_d)^{i_d},\]
where all coefficients are positive.
By Fubini's theorem,
\begin{equation}\label{coef}
  \innerp{Qh}{h}=C_d(\mathrm {det}\ L)^2\sum_j\sum_{k,i_1,...,i_d} C_{k,i_1,...,i_d}\left(\int  g_j(u) u_1^{i_1}\cdots u_d^{i_d}q(u)du\right)^2.
\end{equation}
Thus, $\innerp{Qh}{h} =0$ implies that each term must be zero
\[\int  g_j(u)u_1^{i_1}\cdots u_n^{i_n} q(u) du=0,\ {\rm for\  any\ integers}\ i_1,...,i_d\ge0.\]

Note that the measure $\mu$ defined by $d\mu(u):=q(u)du$ satisfies the condition of Lemma \ref{poly}. Then the polynomials are dense in $L^2(\mathbb{R}^d,q)$. Also, note that $g_j\in L^2(\mathbb{R}^d,q)$ because $|g_j(u)|\le |h((PL)^{-1}u)|$ and
\[
\int |h((PL)^{-1}u)|^2 q(u)du = \E[|h((PL)^{-1}(X'-Y'))|^2] =  \E[|h(X-Y)|^2]= \|h\|_{L^2(\rho)}^2 < \infty,
\]
where in the first equality we used the fact that $q(u)$ is the density of $X'-Y'$. Hence, $g_j=0$ in $L^2(\mathbb{R}^d,q)$ for all $j$, and we conclude that $h=0$ in $L^2(\rho)$.
\end{proof}

\bigskip

\begin{proof}[Proof of Theorem \ref{thm:OU_non-radial}] Similar to the proof of Theorem \ref{thm_main_OU}, we only need to consider the process $\bY^t= \bX^t- \bX_c^t$ because the proof only involves $\br_{ij}^t = \bX^{t}_i- \bX^t_j =\bY^{t}_i- \bY^t_j$.

By Lemma \ref{leamm:Linear}, $(\bY_1^t,\ldots,\bY_N^t)$ is exchangeable Gaussian. In particular, the vector $(\bY^t_1,\bY^t_2,\bY^t_3)$ is exchangeable Gaussian on $\R^{3d}$.   Denote by $p_t(u,v)$ the joint distribution of $(\bY_1^t-\bY_2^t,\bY_1^t-\bY_3^t)$ and recall in \eqref{eq:rhoavg_Rd}, we denote by $\rho_t$ the distribution (with $q_t$ denoting its density) of $\bY_1^t-\bY_2^t$ and denote by $\wbar\rho_T$ the average of $\rho_t$ on $[0,T]$ (with $\wbar q_T$ denoting its density).

By Proposition \ref{Gaussian}, the integral operator $Q_t$ with  kernel  \begin{equation} \label{kernelK2}
K_t(u,v) := \frac{1}{q_t(u)q_t(v)}\frac{ \innerp{u}{v}}{|u||v|} p_t(u,v)
 \end{equation} is strictly positive on $L^2(\rho_t)$. Then it follows from the non-radial version of Proposition \ref{prop.cc2PD_t} that the coercivity condition holds true for each time $t$.

Similar to the proof in Theorem \ref{thm_main_OU}, the coercivity on $[0,T]$ follows from the non-radial version of Proposition \ref{prop.cc2PD} and Proposition \ref{prop:Qavg_Qt}. More precisely, we would like to show that the integral operator $\bar Q_T$ with kernel
 \begin{equation*}
\bar K_T(u,v) := \frac{1}{\wbar q_T(u)\wbar q_T(v)} \frac{\innerp{u}{v}}{|u||v|} \int_0^Tp_t(u,v)dt
 \end{equation*} is strictly positive on $L^2(\bar\rho_T)$. Note that  $\langle \wbar Q_T  h, h\rangle_{L^2(\wbar\rho_T)}  = \frac{1}{T} \int_0^T \langle Q_t h, h\rangle_{L^2(\rho_t)} dt$ for any $h\in L^2(\wbar \rho_T)$, and we just showed above that $\langle Q_t h, h\rangle_{L^2(\rho_t)}>0$ for each $t\in [0,T]$. Thus, $\wbar Q_T$ is strictly positive.
 \end{proof}

\section{Nonlinear systems with three particles}\label{sec:3particleSys}
We consider the class of nonlinear systems with interaction functions in Theorem \ref{thm_mainAll}(c), starting from an invariant measure. Since the diffusion process $(\bX^t, t\in [0,T])$ has a stationary distribution, the coercivity condition on a time instance is the same as on the time interval  $[0,T]$, because $\rho_T=\rho_t$. Thus, we simply say that coercivity condition holds true without specifying it being at a single time or on a time interval when the process is stationary.

Global solutions exist for the gradient system \eqref{eq:sys_grad} with a potential $\Phi(r)$ in \eqref{eq:Phi_nonlinear}, because the potential leads to a locally Lipschitz drift term and the total potential $J_\phi$ in \eqref{eq:potential} leads to a Lyapunov function for the system. 

 The following theorem covers Theorem \ref{thm_mainAll}(c).
\begin{theorem}\label{Thm:cc_general} The coercivity condition holds true for the system \eqref{eq:sys_grad} with $N=3$ starting from an initial condition $\bX^0$ such that the joint density of $(\bX^0_1-\bX^0_2, \bX^0_1-\bX^0_3)$ is $p(u,v)$ in  \eqref{eq:r_invPDF}, and with $\Phi$ in  \eqref{eq:Phi_nonlinear}.
\end{theorem}

Note that we only need the joint density of $(\bX^t_1-\bX^t_2, \bX^t_1-\bX^t_3)$ to study the coercivity condition. We will show first that  $p(u,v)$ in  \eqref{eq:r_invPDF} is an invariant density for $(\bX^t_1-\bX^t_2, \bX^t_1-\bX^t_3)$ when $N=3$ (see Section \ref{sec:stationary}).  Then based on the analytical expression of $p(u,v)$, we introduce a ``comparison to Gaussian kernels'' technique, which makes extensive use of positive-definite kernels. We will prove the theorem and develop the technique in two steps: first when $\Phi(r)=r^{2\beta}$ in Section \ref{sec:powerFn} and then general $\Phi$ in Section \ref{sec:genFn}. Due to the need of the above analytical form of $p(u,v)$, our result is limited to the case when $N=3$ and when the initial distribution is the invariant measure (see Remark \ref{Rmk:N=3}).

\subsection{Stationary distribution for pairwise differences} \label{sec:stationary}

We show first that the process of pairwise differences $(\bX^t_1-\bX^t_2, \bX^t_1-\bX^t_3)$ admits a stationary distribution.
\begin{proposition} \label{prop_staionary_r}
Suppose that $\Phi\in C^2(\R^+,\R)$ and $Z= \int_{\R^d}\int_{\R^d} e^{-H(u,v)} du dv< \infty$ for
\[
H(u,v) =\frac{1}{3} [\Phi(|u|)+ \Phi(|v|)+ \Phi(|u-v|)].
\]
Then the process  $(\br^t_{12},\br^t_{13}) =  (\bX^t_1-\bX^t_2, \bX^t_1-\bX^t_3)$ admits $p(u,v)$ in \eqref{eq:r_invPDF} as an invariant density.

\end{proposition}
\begin{proof}
Note that
\begin{equation} \label{sys:r_ij}
\left \{
\begin{array}{ll}
    d\br^t_{12}& = F(\br^t_{12}, \br^t_{13})dt + (d\bB^t_1-d\bB^t_2),   \\
    d\br^t_{13}& = F(\br^t_{13}, \br^t_{12})dt + (d\bB^t_1-d\bB^t_3)  ,
\end{array}
\right.
\end{equation}
where the function $F:\R^d\times \R^d\to \R^d$ is given by
\[
F(u,v) = - \frac{1}{3} [2\phi(|u|)u +\phi(|v|)v +\phi(|u-v|)(u-v)  ],
\]
where $\phi(r)= \Phi'(r)$.

The diffusion $ \begin{pmatrix} d\bB^t_1-d\bB^t_2\\ d\bB^t_1-d\bB^t_3  \end{pmatrix}$ has a non-degenerate covariance $\begin{pmatrix} 2I_d & I_d\\ I_d& 2I_d  \end{pmatrix}$. One can then verify directly that the distribution $p(u,v)$ is a stationary solution to the Kolmogorov forward equation of \eqref{sys:r_ij}. Alternatively, for ease of computation, we show that the system \eqref{sys:r_ij} is a linear transformation of a gradient system with homogeneous diffusion, which shares the same invariant measure. Let $A = \sqrt{2}\begin{pmatrix} I_d & 0\\ 1/2 I_d& \sqrt{3}/2 I_d  \end{pmatrix}$, which satisfies  $AA^T = \begin{pmatrix} 2I_d & I_d \\ I_d & 2I_d  \end{pmatrix}$. Then the process
\[
\begin{pmatrix} \bY_1^t \\ \bY_2^t  \end{pmatrix}
= A^{-1} \begin{pmatrix} \br^t_{12}\\  \br^t_{13} \end{pmatrix}
\]
 is a weak solution to the system
\begin{align} \label{sys:r_grad}
d \begin{pmatrix} \bY_1^t \\ \bY_2^t  \end{pmatrix}
& = A^{-1} \begin{pmatrix} F(\bY^t_1, \frac{\sqrt{2}}{2}\bY^t_1 + \frac{\sqrt{6}}{2}\bY^t_2) \\ F(\frac{\sqrt{2}}{2}\bY^t_1 + \frac{\sqrt{6}}{2}\bY^t_2), \bY^t_1) \end{pmatrix} dt +  \begin{pmatrix} d\widetilde\bB_1^t\\ d\widetilde\bB_2^t \end{pmatrix},
\end{align}
where $(\widetilde\bB_1^t,\widetilde\bB_2^t)$ is a standard Brownian motion on $\R^{2d}$. Notice that $H(u,v) = H(v,u)$ and
\begin{align*}
 A^{-1} \begin{pmatrix} F(\sqrt{2}u,\, \frac{\sqrt{2}}{2}u + \frac{\sqrt{6}}{2}v) \\ F(\frac{\sqrt{2}}{2}u + \frac{\sqrt{6}}{2}v,\, \sqrt{2}u) \end{pmatrix}  =  \begin{pmatrix} \nabla_u[  H(\sqrt{2}u,\, \frac{\sqrt{2}}{2}u + \frac{\sqrt{6}}{2}v) ] \\ \nabla_v[ H(\frac{\sqrt{2}}{2}u + \frac{\sqrt{6}}{2}v,\, \sqrt{2}u) ] \end{pmatrix}.
\end{align*}
Then, it follows from Lemma \ref{lemma:GradS_inv} that $p_{\bY}(\by_1,\by_2)\propto e^{-2 H(\by_1,\frac{\sqrt{2}}{2}\by_1 + \frac{\sqrt{6}}{2}\by_2))}$ is an invariant density for the system \eqref{sys:r_grad}. Therefore, the process $(\br^t_{12},\br^t_{13})$ admits $p(u,v)$
as an invariant density.
\end{proof}

\begin{remark}
Similarly, one can prove that the process $(\bX^{t}_1- \bX^t_2, \bX^{t}_1- \bX^t_3,\dots,\bX^{t}_1- \bX^t_N)$ admits a stationary density on $\R^{(N-1)d}$. In essence, we decompose the system into a reference particle and the relative positions of other particle to the reference particle. This is similar to the macro-micro decomposition of the system in \cite{malrieu2003convergence,cattiaux2008probabilistic, ha2009emergence,cattiaux2018stochastic}. But the above transformation has the following advantage: it leads to a gradient system with an additive white noise, and this simplifies the derivation of the stationary distribution.  
\end{remark}

\begin{remark}\label{Rmk:N=3}
Our current proof for the coercivity condition makes use of the analytical expression of the invariant density $p(u,v)$ of $(\bX^{t}_1- \bX^t_2, \bX^{t}_1- \bX^t_3)$ and of the M\"untz-type theorem on $\R^+$. When $N>3$, such an explicit form is no longer available due to the need of marginalizing the joint distribution of $(\bX^{t}_1- \bX^t_2, \bX^{t}_1- \bX^t_3,\dots,\bX^{t}_1- \bX^t_N)$:
\begin{equation*}
p(u,v)=\frac{1}{Z}f(u,v)e^{-\frac{2}{N}\left[\Phi(|u|)+\Phi(|v|)+\Phi(|u-v|)\right]},
\end{equation*}
where $Z$ is a normalizing constant and
\begin{align*}
f(u,v)=\int_{\R^{d(N-3)}}
\label{fuv} e^{-\frac{2}{N}\left[\sum_{4\leq i<j}^N\Phi(|\mathbf{r}_{1i}-\mathbf{r}_{1j}|)+\sum_{l=4}^N\left[\Phi(|\mathbf{r}_{1l}|)+\Phi(|u-\mathbf{r}_{1l}|)+\Phi(|v-\mathbf{r}_{1l}|)\right]\right]}     d\mathbf{r}_{14}\ldots\mathbf{r}_{1N}.
\end{align*}
We expect to remove the constraint $N=3$ in future research.  Also, when the system starts from initial distributions other than the invariant measure, the density of $(\bX^{t}_1- \bX^t_2, \bX^{t}_1- \bX^t_3)$ will no longer be $p(u,v)$. Perturbation type arguments may help to address such cases.
\end{remark}

\subsection{Interaction potentials in form of $\Phi(r) = r^{2\beta}$} \label{sec:powerFn}
We develop in this section a ``comparison to Gaussian kernels'' technique to prove that the coercivity condition holds true for systems with interaction potential $\Phi(r) = r^{2\beta}$ for $1/2\leq \beta\leq 1$.
We first prove that $p(u,v)$ in \eqref{eq:r_invPDF} is positive-definite, then prove that the integral kernel $K_t$ is strictly positive-definite by comparing it with the Gaussian kernel.

\begin{lemma}\label{l31}
Assume $\Phi(r)=r^{2\beta}$. Let $p(u,v)$ be the density function defined in \eqref{eq:r_invPDF}.
\begin{enumerate}
\item[(1)] If $0<\beta \leq 1$, all the three kernels, $\Phi(|u-v|)$, $e^{-\Phi(|u-v|)}$, and $p(u,v)$, are positive-definite.
\item[(2)] If $\beta>1$,  then $p(u,v)$ is not positive-definite.
\end{enumerate}
\end{lemma}
\begin{proof}This is a generalization of Corollary 3.3.3 of \cite{BCR84} to the higher-dimensional case. Note that $p(u,v)$ is positive-definite iff $e^{-\Phi(|u-v|)}$ is positive-definite.
The kernel $|u-v|^2$ for $u,v\in\RR^d$ is a negative definite kernel,  because for any $c_1,\ldots,c_n\in\RR$, and $\sum_{i=1}^{n}c_i=0$,
\begin{eqnarray*}
\sum_{i,j=1}^{n}c_ic_j|u_i-u_j|^2&=&\left[\sum_{i=1}^{n}c_i\right]\left[\sum_{j=1}^n c_j |u_j|^2\right]+\left[\sum_{j=1}^{n}c_j\right]\left[\sum_{i=1}^n c_j |u_i|^2\right]-\left|\sum_{i=1}^n c_iu_i\right|^2\\
&=&-\left|\sum_{i=1}^n c_iu_i\right|^2\leq 0.
\end{eqnarray*}
By Theorem \ref{t54}, $|u-v|^{2\beta}$ is also a negative definite kernel for any $0< \beta\leq 1$. By Theorem \ref{t53}, we obtain that $e^{-|u-v|^{2\beta}}$ is positive-definite, then Part (1) follows.

Now we prove Part (2). Suppose now that for some $\beta>1$, $p(u,v)$ is a positive-definite kernel. Then for any $t>0$, $x_1,\ldots,x_n\in \RR^d$ and $c_1,\ldots,c_n\in \RR$, we have
\begin{eqnarray*}
\sum_{j,k=1}^{n}c_jc_ke^{-t|x_j-x_k|^{2\beta}}=\sum_{j,k=1}^{n}c_jc_ke^{-\left|t^{\frac{1}{2\beta}}x_j-t^{\frac{1}{2\beta}}x_k\right|^{2\beta}}\geq 0
\end{eqnarray*}
By Theorem \ref{t53}, the kernel $|u-v|^{2\beta}$ is negative definite, and by Theorem \ref{t55}, $|u-v|^{\beta}$ is a metric on $\RR$. Let $\mathbf{0}=(0,\ldots,0)\in \RR^d$, $\mathbf{1}=(1,\ldots,1)\in \RR^d$ and $\mathbf{2}=(2,\ldots,2)\in \RR^d$. Note that
\begin{eqnarray*}
|\mathbf{0}-\mathbf{1}|^{\beta}=d^{\frac{\beta}{2}}, \quad
|\mathbf{0}-\mathbf{2}|^{\beta}=2^{\beta}d^{\frac{\beta}{2}}> 2 |\mathbf{0}-\mathbf{1}|^{\beta}
\end{eqnarray*}
when $\beta>1$. The contradiction to the triangle inequality implies Part (2).
\end{proof}


\bigskip
Recall that the coercivity condition depends only on the distribution of the process $(\br_{12}^t,\br_{13}^t)$. When the process $(\br_{12}^t,\br_{13}^t)$ is stationary, the coercivity condition at a time instance in Definition \ref{def_coercivty_t} is equivalent to that of Definition \ref{def_coercivty}. Following proposition \ref{prop.cc2PD}, the coercivity condition is equivalent to the positivity of the integral operator $Q$ on $L^2(\rho)$ with kernel $K(r,s):\R^+\times \R^+\to\R$ defined by
  \begin{equation} \label{kernelK_inv}
K(r,s) := \frac{1}{q (r)q(s)} (rs)^{d-1}\int_{S^{d-1}}\int_{S^{d-1}}\innerp{\xi}{\eta} p(r\xi,s\eta) d\xi d\eta,
 \end{equation}
 where  $p(u,v)$ is the stationary density defined in \eqref{eq:r_invPDF}, and  $\rho$ denotes the distribution of $|\br_{12}^t|$ with $q$ being its density. For the case $\beta=1$ in the previous section, we witnessed that the Gaussian distribution neatly ensures strict positivity of the integral operator through Taylor expansion of $\inp{u,v}e^{c\inp{u,v}}$. However, when $\beta\neq 1$, such a ``quadratic structure'' of the Gaussian kernel is no longer available. Using the positive-definiteness of $-|u-v|^{2\beta}$, we uncover a quadratic structure by bounding the kernel $K$ by another positive-definite kernel from below and by a Gamma integral representation of the power function. We call this procedure a ``comparison to a Gaussian kernel" technique.

 We start with two inequalities using positive-definiteness of integral kernels.
 \begin{lemma}\label{lemma_exp_bound}
 Let $\Phi_i:\R^d\times \R^d\to \R$ be positive-definite kernels for $i=1,2$.  Then,
 \begin{align*}
 &  \int_{\R^d} \int_{\R^d} h(u)h(v) \Phi_1(u,v) e^{\Phi_2(u,v)} dudv
 \geq  \int_{\R^d} \int_{\R^d} h(u)h(v) \Phi_1(u,v)  dudv, \\
&  \int_{\R^d} \int_{\R^d} h(u)h(v) \Phi_1(u,v) e^{\Phi_2(u,v)} dudv
 \geq  \int_{\R^d} \int_{\R^d} h(u)h(v) \Phi_1(u,v) \Phi_2(u,v)  dudv
 \end{align*}
for any $h\neq 0$, as long as the integrals exist.
 \end{lemma}
 \begin{proof}
 By Theorem \ref{t52}, $\Phi_2(u,v)^n \Phi_1(u,v) $ is positive-definite for each integer $n\geq 0$, so for any $h\neq 0$ and $n\geq 0$, we have $ \int_{\R^d} \int_{\R^d} h(u)h(v) \Phi_1(u,v) \Phi_2(u,v)^ndudv >0$. Then the inequalities follow from the Taylor expansion of  $e^{\Phi_2(u,v)}$, with the first inequality keeping only the term with $n=0$ and the second inequality keeping only the term with $n=1$.
  \end{proof}


The following proposition is a counterpart of Lemma \ref{Gaussian}.
\begin{proposition}\label{prop_powerKernel}
Let $\beta\in (0,1]$ and  $p(u,v)$ be a density function in \eqref{eq:r_invPDF} with $\Phi(r)= r^{2\beta}$, i.e., $p(u,v) = \frac{1}{Z} e^{-\frac{2}{3} (|u|^{2\beta}+ |v|^{2\beta}+|u-v|^{2\beta})}$. Let  $\rho(r)$ be the distribution of $|U|$ with $(U,V)$ having a joint distribution $p(u,v)$. Then, for any $0\neq h\in L^2(\rho)$,
 \begin{align*}
I 
&=\int_{\R^d}\int_{\R^d}h(|u|)h(|v|)\frac{\langle u,v \rangle}{|u||v|}p(u,v)dudv >0.
\end{align*}

\end{proposition}

 \begin{proof} The factor $\frac{2}{3}$ and the normalizing constant $Z$ do not play a role in the above inequality, so we omit them in the following proof. We only consider the case $\beta<1$, since when $\beta=1$,  the Gaussian distribution neatly ensures strict positivity of the integral operator through Taylor expansion of $\inp{u,v}e^{c\inp{u,v}}$. Note that
 \begin{align*}
I &=\int_{\R^d}\int_{\R^d}h(|u|)e^{-2|u|^{2\beta}}h(|v|)e^{-2|v|^{2\beta}}\frac{\langle u,v \rangle}{|u||v|}e^{|u|^{2\beta}+ |v|^{2\beta}-|u-v|^{2\beta}}dudv\\
&= \int_{\R^d}\int_{\R^d} \widetilde h(|u)\widetilde h(|v|)\frac{\langle u,v \rangle}{|u||v|}e^{\Phi_2(u,v)}dudv,
 \end{align*}
 where $\widetilde h(r) :=h(r)e^{-2r^{2\beta}}$ and
\begin{eqnarray*}
\Phi_2(u,v) : =|u|^{2\beta}+|v|^{2\beta}-|u-v|^{2\beta}.
\end{eqnarray*}
By Lemma \ref{l31}, $|u-v|^{2\beta}$ is negative definite. Then, by Theorem \ref{tpn}, $\Phi_2(u,v)$ is positive-definite. Thus, by Lemma \ref{lemma_exp_bound} with $\Phi_1(u,v) = \inp {u,v}$ and $\Phi_2(u,v)$ as above, we have
\begin{align}
I &\geq \int_{\RR^d}\int_{\RR^d}\widetilde h(|u)\widetilde h(|v|)\frac{\langle u,v \rangle}{|u||v|}\left(|u|^{2\beta}+|v|^{2\beta}-|u-v|^{2\beta}\right)dudv
\nonumber \\
&=-\int_{\RR^d}\int_{\RR^d}\widetilde h(|u)\widetilde h(|v|)\frac{\langle u,v \rangle}{|u||v|}|u-v|^{2\beta}dudv =:\widetilde I  , \label{eq:|u-v|2beta}
\end{align}
where in the equality we dropped the term $|u|^\beta+|v|^\beta$, due to symmetry of $\langle u,v \rangle$,
\[\int_{\R^d}\int_{\R^d}g_1(|u|)g_2(|v|)\frac{\langle u,v \rangle}{|u||v|} =0 \] for any $g_1,g_2\in L^2(\rho)$. We shall use this property several times in the following.

Next, we use Gamma function to bound 
$\widetilde I $ in \eqref{eq:|u-v|2beta} from below by a Gaussian kernel as in the previous section.
Note that for any $x>0$ and $\beta< 1$, 
\[
|u-v|^{2\beta} = \frac{\beta}{\Gamma(1-\beta)}\int_0^{\infty}(1-e^{-\lambda |u-v|^{2}})\frac{d\lambda}{\lambda^{\beta+1}}.
\]
Plugging this into the integral in \eqref{eq:|u-v|2beta}, and using the symmetry of $\langle u,v\rangle$ again, we obtain
\begin{eqnarray*}
\widetilde I  &= & \frac{\beta}{\Gamma(1-\beta)}\int_0^{\infty}\int_{\RR^d}\int_{\RR^d}\widetilde h(|u)\widetilde h(|v|)\frac{\langle u,v \rangle}{|u||v|}(e^{-\lambda|u-v|^2}-1)dudv \frac{d\lambda}{\lambda^{\beta+1}}\\
&=&\frac{\beta}{\Gamma(1-\beta)}\int_0^{\infty}\int_{\RR^d}\int_{\RR^d}\widetilde h(|u)\widetilde h(|v|)\frac{\langle u,v \rangle}{|u||v|}e^{-\lambda|u-v|^2}dudv \frac{d\lambda}{\lambda^{\beta+1}}
\\ &=&\frac{\beta}{\Gamma(1-\beta)}\int_0^{\infty}\int_{\RR^d}\int_{\RR^d}\widetilde h(|u|) \widetilde h(|v|) e^{-\lambda (|u|^2 +|v|^2) } \frac{\langle u,v \rangle}{|u||v|} e^{2\lambda \langle u,v \rangle }dudv \frac{d\lambda}{\lambda^{\beta+1}}.
\end{eqnarray*}
By the symmetry of $\langle u,v\rangle $ and Taylor expansion of $e^{2\lambda \langle u,v \rangle }$, we have
\begin{eqnarray*}
\widetilde I  &= & \frac{\beta}{\Gamma(1-\beta)}\int_0^{\infty}\int_{\RR^d}\int_{\RR^d}\widetilde h(|u|) \widetilde h(|v|)e^{-\lambda (|u|^2 +|v|^2) }  \frac{\langle u,v \rangle}{|u||v|}\sum_{n=0}^{\infty}\frac{\langle u,v \rangle^{2n+1}}{(2n+1)!}dudv 2^{2n+1}\lambda^{2n-\beta}d\lambda \\
&=& \sum_{n=0}^{\infty}\frac{\beta}{\Gamma(1-\beta)}\int_0^{\infty}\int_{0}^{\infty}\int_{0}^{\infty}\widetilde h(r) \widetilde h(s) e^{-\lambda (r^2 +s^2) }(rs)^{d+2n+1} C_n dr ds \lambda^{2n-\beta} d\lambda,
\end{eqnarray*}
where we denote $C_n := \int_{\xi\in S^{d-1}}\int_{\eta\in S^{d-1}}\frac{2^{2n+1}\langle \xi,\eta \rangle^{2n+2}}{(2n+1)!} d\xi d\eta$,
which is positive. Thus,
\begin{eqnarray*}
\widetilde I  &= &\sum_{n=0}^{\infty}\frac{C_n\beta}{\Gamma(1-\beta)}\int_0^{\infty}\lambda^{2n-\beta}\left[\int_{0}^{\infty}\widetilde h(r)e^{-\lambda r^2 } r^{d+2n+1}dr\right]^2 d\lambda\\
&\geq&\sum_{n=0}^{\infty}\frac{C_n\beta}{\Gamma(1-\beta)}\int_1^2\lambda^{2n-\beta} d\lambda \left[\int_{0}^{\infty}\widetilde h(r) e^{-\lambda r^2 }r^{d+2n+1}dr\right]^2.
\end{eqnarray*}
Note that
$\tilde{C}_{n,\beta}=\frac{C_n\beta}{\Gamma(1-\beta)}\int_1^2\lambda^{2n-\beta} d\lambda>0$ for each $n\geq 0$. Combing all the above, we have
\begin{eqnarray*}
I\geq \sum_{n=0}^{\infty}\tilde{C}_{n,\beta}\left[\int_{0}^{\infty}h(r)e^{-2r^{2\beta}-2 r^2}r^{d+2n+1}dr\right]^2,
\end{eqnarray*}
which is positive if $h\neq 0 \in L^2(f)$ with $f(r):=r^{d+1}e^{-2r^{2\beta}-2 r^2}$, because by Lemma \ref{lemma:Muntz}, the set of functions ${\rm span}\{1,r^2, r^4,\cdots\}$ is complete in $L^2(\R^+, f)$. Note that $\mathrm{supp}\,f = \mathrm{supp }\,\rho = \R^+$, so $h\neq 0 \in L^2(f)$ when $h\neq 0 \in L^2(\rho)$.
\end{proof}

\bigskip
\begin{proof}[Proof of Theorem \ref{Thm:cc_general} with $\Phi(r)=r^{2\beta}$] Since $\beta\in [1/2,1]$ and the density of $(\br_{12}^0,\br_{13}^0)$ is $p(u,v)$, it follows from Proposition \ref{prop_staionary_r} that the process $(\br_{12}^t,\br_{13}^t)$ is stationary  with distribution $p(u,v)$. Then, the coercivity condition is equivalent to that
\[
I : = \E[h(|\br_{12}^t|)h(|\br_{13}^t|)\frac{\innerp{\br_{12}^t}{\br_{13}^t}}{|\br_{12}^t||\br_{13}^t|}]>0
\]
for any $h\neq 0\in L^2(\rho)$, where $\rho$ is the stationary probability density of $|\br_{12}^t|$.
Note that
 \begin{align*}
I 
&=\frac{1}{Z}\int_{\R^d}\int_{\R^d}h(|u|)h(|v|)\frac{\langle u,v \rangle}{|u||v|}e^{-\frac{2}{3} (|u|^{2\beta}+ |v|^{2\beta}+|u-v|^{2\beta}) }dudv.
\end{align*}
Then we can conclude the theorem by Proposition \ref{prop_powerKernel}.
\end{proof}

\begin{remark}
We point out that the requirement $\beta \in (0,1]$ is to ensure that the stationary density $p(u,v)$ is a positive-definite kernel. When $\beta >1$, the above method does no longer work, because $p(u,v)$ is not positive-definite as shown in Lemma {\rm \ref{l31}}. The requirement $\beta \geq \frac{1}{2}$ is to ensure that the drift term is continuous, so that a strong solution exists. When $\beta <\frac{1}{2}$, the drift is moderately singular, the existence of a solution is open {\rm \cite{skorokhod1996_RegularityManyparticle,albeverio2003_StrongFeller}}.
\end{remark}

\subsection{General interaction potentials}\label{sec:genFn}
The ``comparison to a Gaussian kernel'' technique in Lemma \ref{lemma_exp_bound} and Proposition \ref{prop_powerKernel} can be generalized to prove the coercivity condition for a large class of interaction functions. The following lemma provides the key element in such a generalization.

\begin{lemma}\label{lemma:transform_to_positve}
Let $\Phi$ be a potential in \eqref{eq:Phi_nonlinear} and let $\rho(r)$ be the distribution of $|U|$ with $(U,V)$ having a joint distribution $p(u,v)= \frac{1}{Z} e^{-\frac{2}{3}  [\Phi(|u|)+ \Phi(|v|)+\Phi(|u-v|) ]  }$.
Then,
 \begin{align*}
I 
&=\int_{\R^d}\int_{\R^d}h(|u|)h(|v|)\frac{\langle u,v \rangle}{|u||v|}e^{-[\Phi(|u|)+ \Phi(|v|)+\Phi(|u-v|) ] }dudv >0
\end{align*}
for any $0\neq h\in L^2(\rho)$.
\end{lemma}
\begin{proof}
Rewrite the integral as
\begin{align*}
I  &=\int_{\R^d}\int_{\R^d}[h(|u|)e^{-\frac{2}{3}\Phi(|u|)}][h(|v|)e^{-\frac{2}{3}\Phi(|v|)}]\frac{\langle u,v \rangle}{|u||v|}e^{-\frac{2}{3}a|u-v|^{2\beta}-\frac{2}{3}\Phi_0(|u-v|)}dudv\\
&=\int_{\R^d}\int_{\R^d}\widetilde h(|u|)\widetilde h(|v|)\frac{\langle u,v \rangle}{|u||v|}e^{-\frac{2}{3}a |u-v|^{2\beta}+\frac{2}{3}\widetilde \Phi(|u-v|)}dudv,
\end{align*}
where $\widetilde h(r)= h(r)e^{-\frac{2}{3}\Phi(r) -\frac{2}{3}\Phi_0(r)}$ and
\begin{eqnarray*}
\widetilde{\Phi}(u,v):=\Phi_0(|u|)+\Phi_0(|v|)-\Phi_0(|u-v|).
\end{eqnarray*}
Since $\Phi_0(|u-v|)$ is negative definite, by Theorem \ref{pdm}, $\widetilde{\Phi}(u,v)$ is positive-definite. Also, by Lemma \ref{l31}, $\inp{u,v}e^{-\frac{2}{3}a |u-v|^{2\beta} }$ is positive-definite.
Hence, by Lemma \ref{lemma_exp_bound}, we have
\begin{align*}
I &\geq \int_{\R^d}\int_{\R^d}\widetilde h(|u|)\widetilde h(|v|)\frac{\langle u,v \rangle}{|u||v|}e^{-\frac{2}{3}a |u-v|^{2\beta}}dudv.
\end{align*}
Then the strictly positive-definiteness follows from Proposition \ref{prop_powerKernel}.
\end{proof}


\bigskip
We can now conclude the proof of Theorem  \ref{Thm:cc_general}. \\
\begin{proof}[Proof of Theorem \ref{Thm:cc_general} ] The case when $\Phi(r) = r^{2\beta}$ is proved in the previous section.
For general potential $\Phi$, the proof is similar. In fact, since $\beta\in[1/2,1]$ and $\Phi_0$ is smooth, the process $(\br_{12}^t,\br_{13}^t)$ is stationary with invariant density $p(u,v)$.  Then, $\E[h(|\br_{12}^t|)h(|\br_{13}^t|)\frac{\innerp{\br_{12}^t}{\br_{13}^t}}{|\br_{12}^t||\br_{13}^t|}]$ is the same as the integral $I$ in Lemma \ref{lemma:transform_to_positve}, so the coercivity condition follows. We conclude the proof.
\end{proof}

\bigskip
We provide a few examples of negative definite radial kernels and the related positive kernels.
\begin{lemma}\label{lem:examplesND}
For $0<\alpha\leq 2$, $0<\gamma\leq 1$ and $a\geq 0$, the following kernels are negative definite:
\begin{align*}
\Phi_1(|u-v|) &= (a+|u-v|^{\alpha})^{\gamma};\\
\Phi_2(|u-v|) &= \log [1+ (a+|u-v|^{\alpha})^{\gamma}].
\end{align*}
For any $c>0$ and any integer $k\geq 1$, the following kernels are positive-definite:
\[ e^{-c\Phi_1(|u-v|)}, \, e^{-c \Phi_2(|u-v|)},\, \Phi_2(|u-v|)^{-k}. \]

\end{lemma}

\begin{proof}By Lemma \ref{l31}, if $0\leq \alpha \leq 2$, then $|u-v|^{\alpha}$ is a negative definite kernel. By definition of a negative definite kernel $a+|u-v|^{\alpha}$ is also negative definite for any $a\in \RR$.
By Theorem \ref{t54}, $\Phi_1(|u-v|) = (a+|u-v|^{\alpha})^{\gamma}$ is also a negative definite kernel when $0<\gamma\leq 1$ and $a\geq 0$.

Since $\Phi_1(|u-u|)=a^{\gamma}\geq 0$, by Theorem \ref{t54}, $\log(1+\Phi_1(|u-v|))=\Phi_2(|u-v|)$ is negative definite.

The positive-definiteness of $e^{-c\Phi_1(|u-v|)}$ and  $e^{-c \Phi_2(|u-v|)}$ follows directly from Theorem \ref{t53}.
The kernel $ \Phi_2(|u-v|)^{-k}$ is positive-definite because
\begin{eqnarray*}
\int_{0}^{\infty}e^{-s \Phi_2(|u-v|) }  ds =\frac{1}{ \Phi_2(|u-v|)}
\end{eqnarray*}
and because that the product of positive-definite kernels are positive-definite.
\end{proof}

\begin{proposition} Assume that the series
\begin{eqnarray}
\Phi_1(r)&=&c_0+\sum_{j=1}^{\infty}c_j\log\left[1+(a_j+r^{\alpha_j})^{\gamma_j}\right]
-\sum_{j=-1}^{-\infty}c_j[\log(1+(a_j+r^{\alpha_j})^{\gamma_j})]^{-k_j}\label{ps1}\\
\Phi_2(r)&=&\sum_{i=1}^{\infty}c_i'[(a_i'+r^{\alpha'_i})^{\gamma_i'}]-\sum_{i=-1}^{-\infty}c_i'[1+(a_i'+r^{\alpha'_i})^{\gamma_i'}]^{-\beta i}\label{ps2}
\end{eqnarray}
converge for every $r\in \R^+$, where the coefficients satisfy the following conditions
\begin{enumerate}
\item $a_j\geq 0, a'_i\geq 0, c_j\geq 0, c'_i\geq 0$ for $i,j\neq 0$ and
\item $0<\gamma_i\leq 1$, $ \alpha_j, \alpha'_i\in [1,2]$ for $i,j\neq 0$, and
\item $\beta_j>0$ and $k_j\geq 1$ is a positive integer for each $j$.
 \end{enumerate}
 Let $K: \R^+\times \R^+\to\R$ be an integral kernel defined in \eqref{kernelK_inv} with $p(u,v)$ defined in \eqref{eq:r_invPDF} and with
 \begin{eqnarray}
\Phi(r)&=&\Phi_1(r)+\Phi_2(r). \label{ph}
\end{eqnarray}
Then  $K(r,s)$ is a positive-definite kernel. Furthermore, if there exists $i_0\geq 1$, such that
\begin{align}
a_{i_0}'=0,\ \gamma_{i_0}'=1,\ \mathrm{and}\ c_{i_0}'>0,\label{ccb}
\end{align}
then the coercivity condition holds true on $L^2(\rho)$ for the system \eqref{eq:sys_grad} with potential $\Phi$ in \eqref{ph}, if it starts from $(\br_{12}^0,\br_{13}^0)$ with a joint density $p(u,v)$ in \eqref{eq:r_invPDF}, where $\rho$ is the distribution of $|\br_{12}^0|$.
\end{proposition}
\begin{proof}
It follows directly from Lemma \ref{lem:examplesND} that $K$ is positive-definite. Note that with the above conditions, the drift term is smooth and dominated by a term $r^{2\beta}$ with $\beta= \alpha_{i_0}'/2 \in [1/2,1]$, so the system leads to a stationary process $(\br_{12}^t,\br_{13}^t)$.  It follows from Lemma \ref{lemma:transform_to_positve} that the coercivity condition holds true.
\end{proof}

\appendix
\renewcommand*{\thesection}{\Alph{section}}
\section{Appendix}\label{sec:append}
\subsection{Positive-definite integral kernels}

In this section, we review the definitions of positive and negative definite kernels, as well as their basic properties. The following definition is a real version of the definition in \cite[p.67]{BCR84}.

\begin{definition}\label{def_spd}
Let $\mathbb{X}$ be a nonempty set. A function $K:  \mathbb{X}\times \mathbb{X}\rightarrow \RR$ is called a (real) positive-definite kernel iff it is symmetric (i.e. $K(x,y)=K(y,x)$) and
\begin{eqnarray}
\sum_{j,k=1}^{n}c_jc_k K(x_j,x_k)\geq 0\label{pr}
\end{eqnarray}
for all $n\in \NN$, $\{x_1,\ldots,x_n\}\in \mathbb{X}$ and $\{c_1,\ldots,c_n\}\in \RR$. We call the function $\phi$ a (real) negative definite kernel iff it is symmetric and
\begin{eqnarray}
\sum_{j,k=1}^{n}c_jc_k  K(x_j,x_k)\leq 0\label{nr}
\end{eqnarray}
for all $n\geq 2$, $\{x_1,\ldots,x_n\}\in \mathbb{X}$ and $\{c_1,\ldots,c_n\}\in \RR$ with $\sum_{j=1}^{n}c_j=0.$
\end{definition}

\noindent{\textbf{Remark.}} In \cite[p.67]{BCR84}, a function $K: \mathbb{X}\times \mathbb{X}\rightarrow \CC$ is defined to be positive-definite iff
\begin{eqnarray}
\sum_{j,k=1}^{n}c_j\overline{c}_k K(x_j,x_k)\geq 0\label{pc}
\end{eqnarray}
for all $n\in \NN$, $\{x_1,\ldots,x_n\}\in \mathbb{X}$ and $\{c_1,\ldots,c_n\}\in \CC$, where $\overline{c}$ denotes the complex conjugate of a complex number $c$. It is straightforward to check that when $\phi$ is real-valued and symmetric, the definitions (\ref{pr}) and (\ref{pc}) are equivalent. Similarly, In the definition of negative definiteness in \cite[p.67]{BCR84}, a function $K: \mathbb{X}\times \mathbb{X}\rightarrow \CC$ is negative definite iff it is Hermitian (i.e. $ K(x,y)=\overline{ K(y,x)}$) and
\begin{eqnarray}
\sum_{j,k=1}^{n}c_j\overline{c}_k K(x_j,x_k)\leq 0\label{nc}
\end{eqnarray}
for all $n\geq 2$, $\{x_1,\ldots,x_n\}\in \mathbb{X}$ and $\{c_1,\ldots,c_n\}\in \CC$ with $\sum_{j=1}^{n}c_j=0$. We can again check that when $\phi$ is real-valued, the definitions (\ref{nr}) and (\ref{nc}) are equivalent. In this paper, we only consider real-valued, symmetric kernels.

 \begin{theorem}[Properties of positive-definite kernels]\label{t52}
 Suppose that $k, k_1, k_2: \mathbb{X}\times\mathbb{X}\subset\mathbb{R}^d\times\mathbb{R}^d\to \mathbb{R}$ are positive-definite kernels. Then
\begin{enumerate} \setlength\itemsep{0mm}
\item $c_1k_1+c_2k_2$ is positive-definite, for $c_1,c_2\ge0$

\item $k_1k_2$ is positive-definite. (\cite[p.69]{BCR84})

\item $\exp(k)$ is positive-definite. (\cite[p.70]{BCR84})

\item $k(f(u),f(v))$ is positive-definite for any  map $f:\mathbb{R}^d\to \mathbb{R}^d$

\item Inner product $\inp{u,v}=\sum_{j=1}^du_jv_j$ is positive-definite (\cite[p.73]{BCR84})

\item $f(u)f(v)$ is positive-definite for any function $f:\mathbb{X}\to \mathbb{R}$ (\cite[p.69]{BCR84}).

\item If $k(u,v)$ is measurable and integrable, then $\iint k(u,v)dudv\ge0$ (\cite[p.524]{RKSF})
\end{enumerate}
\end{theorem}

\begin{theorem} \rm{ \cite[Theorem 3.1.17]{BCR84}} \label{pdm}
Let $K: \mathbb{X}\times \mathbb{X}\rightarrow \RR$ be symmetric. Then $K$ is positive-definite iff
\begin{align*}
\det( K(x_j,x_k)_{j,k\leq n})\geq 0
\end{align*}
for all $n\in \NN$ and all $\{x_1,\ldots,x_n\}\subseteq \mathbb{X}$.
\end{theorem}

\begin{theorem}\rm{ \cite[Lemma 3.2.1]{BCR84} } \label{tpn}
Let $\mathbb{X}$ be a nonempty set, $x_0\in \mathbb{X}$ and let $\psi: \mathbb{X}\times \mathbb{X}\rightarrow \RR$ be a symmetric kernel. Put $ K(x,y):=\psi(x,x_0)+\psi(y,x_0)-\psi(x,y)-\psi(x_0,x_0)$. Then $K$ is positive-definite iff $\psi$ is negative definite.
\end{theorem}

\begin{theorem}\label{t53}Let $\mathbb{X}$ be a nonempty set and let $\psi: \mathbb{X}\times \mathbb{X}\rightarrow \mathbb{R}$ be a kernel. Then $\psi$ is negative definite iff $\exp(-t\psi)$ is positive-definite for all $t>0$.
\end{theorem}\vspace{-2mm}
\begin{proof}The complex version of this theorem is proved in Theorem 3.2.2 of  of \cite{BCR84}. The real version can be proved in a similar way.
\end{proof}

\begin{theorem}\label{t54}If $\psi:\mathbb{X}\times \mathbb{X}\rightarrow\RR$ is negative definite and $\psi(x,x)\geq 0$, then so are $\psi^{\alpha}$ for $0<\alpha<1$ and $\log(1+\psi)$.
\end{theorem}\vspace{-2mm}
\begin{proof}The complex version of this theorem is proved in Theorem 3.2.10 of \cite{BCR84}. The real version can be proved in a similar way.
\end{proof}

\begin{theorem}\rm{\cite[Proposition 3.3.2]{BCR84} }\label{t55}Let $\mathbb{X}$ be nonempty and $\psi: \mathbb{X}\times \mathbb{X}\rightarrow \CC$ be negative definite. Assume $\{(x,y)\in \mathbb{X}\times \mathbb{X}, \psi(x,y)=0\}=\{(x,x): x\in \mathbb{X}\}$, then $\sqrt{\psi}$ is a metric on $\mathbb{X}$.
\end{theorem}




\subsection{M\"untz-type theorems on half-line}
We recall first the following theorem on the completeness of $\{t^{a_n}\}$ in weighted $L^2$ space on unbounded domain (see \cite{Fuchs1946, boas1946properties} and see \cite{gao2005generalized, horvath2014muntz} for recent developments ).
\begin{theorem}\label{t510}
Let $a_k$ be positive numbers, such that $a_{k+1}-a_k\geq d>0, (k=1,2,\ldots)$, and let
\begin{eqnarray*}
\log \psi(r)=\begin{cases}2\sum_{a_k<r}\frac{1}{a_k},\ \mathrm{if}\ r>a_1\\\frac{2}{a_1},\ \mathrm{if}\ r\leq a_1.\end{cases}
\end{eqnarray*}
Then $\{e^{-t}t^{a_k}\}$ is complete in $L^2(0,\infty)$ iff
$\int_1^{\infty}\frac{\psi(r)}{r}dr=\infty$.
\end{theorem}

\begin{lemma}\label{lemma:Muntz}
The set of functions $\{r^{2k}, k=1,2,\cdots\}$ is complete in $L^2([0,\infty), \rho)$ for any probability density $\rho$ such that $\sup_{r>0} \rho(r)e^{2r}<\infty$.
\end{lemma}\vspace{-2mm}
\begin{proof} Let $a_k=2k$ for $k=1,2,\cdots$.
We define the function $\log(\psi(r))=2\sum_{a_k<r}\frac{1}{a_k}$, if $r>a_1$, and $\log(\psi(r))=\frac{2}{a_1}$ if $r\leq a_1$. Note that $2\sum_{a_k<r}\frac{1}{a_k} = \sum_{k=1}^{\lfloor{r/2 \rfloor}} \frac{1}{k} >  \ln(\lfloor{r/2 \rfloor})$. Then $\psi(r)\geq r$ and
$\int_{1}^{\infty}\frac{\psi(r)}{r^2}=\infty.$
We conclude that  $\{e^{-t}t^{2k}, k=1,2,\cdots\}$ is complete in $L^2(0,\infty)$ by Theorem \ref{t510}.

To show that $\{r^{2k}, k=1,2,\cdots\}$ is complete in $L^2(\rho)$, assume that
$\langle h (r),r^{2k}\rangle_{L^2(\rho)}=0$ for all $k\geq 1$.  Then
\[ \int_{0}^{\infty}h (r)\rho(r)e^{r} r^{2k} e^{-r}dr = \int_{0}^{\infty}h (r)r^{2k}\rho(r)dr=0\]
 for all $k$. This implies that $h(r)\rho(r) e^{r}=0$ in $L^2[0,\infty)$ (note that  $h(r)\rho(r) e^{r} \in L^2[0,\infty) $ because $\sup_{r>0}\rho(r)e^{2r}<\infty$).  Hence $h(r)\rho(r)=0$ almost everywhere, and $h=0$ in $L^2([0,\infty),\rho)$.
\end{proof}

\subsection{Stationary measure for a gradient system}
\begin{lemma}\label{lemma:GradS_inv}
Suppose $H:\R^n\to \R$ is locally Lipschitz and $\lim_{|x|\to\infty} H(x) = +\infty$ so that $Z= \int_{\R^n}e^{-2H(x)} dx<\infty$. Then $p(x)= \frac{1}{Z} e^{-2H(x)}$ is an invariant density to gradient system
\begin{equation*}
dX_t= -\nabla H(X_t)dt +dB_t,
\end{equation*}
where $(B_t) $ is an n-dimensional standard Brownian motion.
\end{lemma}\vspace{-2mm}
\begin{proof}
It follows directly by showing that $p(x)$ is a stationary solution to the Kolmogorov forward equation,
\[\frac{1}{2} \Delta p + \nabla \cdot (p \nabla H) = 0.
\]
\end{proof}

\bigskip

\noindent \textbf{Acknowledgements.} {ZL is grateful for support from NSF-1608896 and Simons-638143; FL and MM are grateful for partial support from NSF-1913243; FL for NSF-1821211; MM for NSF-1837991, NSF-1546392 and AFOSR-FA9550-17-1-0280; ST for NSF-1546392 and AFOSR-FA9550-17-1-0280 and AMS Simons travel grant.  FL would like to thank Professor Yaozhong Hu and Dr.~Yulong Lu for helpful discussions. }


\section*{References}
\bibliography{ref_FeiLU,learning_dynamics,LearningTheory,ref_stocParticleSys}


\end{document}